\documentclass[11pt]{amsart}
\usepackage{amsmath, amssymb}
\usepackage{amsfonts}
\usepackage{mathrsfs}
\usepackage[arrow,matrix,curve,cmtip,ps]{xy}

\usepackage{amsthm}

\allowdisplaybreaks

\newtheorem{theorem}{Theorem}[section]
\newtheorem{lemma}[theorem]{Lemma}
\newtheorem{proposition}[theorem]{Proposition}
\newtheorem{corollary}[theorem]{Corollary}
\newtheorem{definition}[theorem]{Definition}

\newtheorem*{theorem*}{Theorem}
\theoremstyle{remark}

\numberwithin{equation}{section}

\begin{document}
\title[$L^p$ estimates for a singular integral operator]{$L^p$ estimates for a singular integral operator motivated by Calder\'{o}n's second commutator}

\author{Eyvindur Ari Palsson}

\address{Department of Mathematics \\ Cornell University \\ Ithaca, NY 14850 \\ USA}
\email{eap48@cornell.edu}


\date{\today}



\begin{abstract}

We prove a wide range of $L^p$ estimates for a trilinear singular integral operator motivated by dropping one average in Calder\'{o}n's second commutator. For comparison by dropping two averages in Calder\'{o}n's second commutator one faces the trilinear Hilbert transform. The novelty in this paper is that in order to avoid difficulty of the level of the trilinear Hilbert transform, we choose to view the symbol of the operator as a non-standard symbol. The methods used come from time-frequency analysis but must be adapted to the fact that our symbol is non-standard.

\end{abstract}

\maketitle

\section{Introduction}

\subsection{History}

The $k$-th Calder\'{o}n commutator, $k\in\lbrace 1,2,3,\ldots\rbrace$, is given by
$$ \mathcal{C}_{A}^{(k)}f(x) = p.v. \int_{\mathbb{R}}\frac{1}{x-y}\left(\frac{A(x)-A(y)}{x-y} \right)^{k}f(y)dy $$
where $A$ is Lipschitz and $A' \in L^{\infty}(\mathbb{R})$. Calder\'{o}n studied these operators in connection with an algebra of pseudo-differential operators. He was also motivated by possible applications to operators of the type

\begin{equation}\label{CauchyIntegral}
\ p.v.\int\limits_{\mathbb{R}}\frac{1}{x-y}\ F\left(\frac{A(x)-A(y)}{x-y}\right) f(y)\ dy
\end{equation}

\noindent where $F$ is an analytic function. The Cauchy integral on Lipschitz curves and double layer potentials are examples of the previous operator. In 1965 Calder\'{o}n showed

\begin{equation*}
\mathcal{C}_{A}^{(k)} : L^p \rightarrow L^p \text{ for } 1<p<\infty
\end{equation*}

\noindent for $k=1$ \cite{C}. Coifman and Meyer extended his result in 1975 to $k=2,3,\ldots$ \cite{CM1}. The estimates obtained did not clearly indicate how the boundedness constant depended on $k$. Building on the work of Coifman and Meyer, Calder\'{o}n was able to prove the above estimates with a boundedness constant that depended on $k$ exponentially. This way he was able to prove bounds for operators of the type (\ref{CauchyIntegral}), as long as the Lipschitz constant was small. Finally, in 1982 Coifman, McIntosh and Meyer showed the above estimates with a boundedness constant that depended on $k$ polynomially \cite{CMM1} and were thus able to show a wide range of $L^p$ estimates for operators of the type (\ref{CauchyIntegral}).

\subsection{Motivation}

Calder\'{o}n observed that one can write the following as an average

$$ \frac{A(x)-A(y)}{x-y} = \int_{0}^{1}A'(x+\alpha(y-x))d\alpha. $$

\noindent Using this trick and a substitution he rewrote his first commutator as

$$ \mathcal{C}_{A}^{(1)}f(x) = \int_{0}^{1}\int_{\mathbb{R}}A'(x+\alpha t)f(x+t)\frac{1}{t}dt\smallskip  d\alpha .$$

\noindent He then asked if one dropped the average and fixed $\alpha$ whether $L^p$ estimates could be found for the resulting operator, uniformly in $\alpha$. This motivated the definition of the bilinear Hilbert transform

$$ BHT_{\alpha}(f_1, f_2)(x) = p.v. \int\limits_{\mathbb{R}}f_1(x+\alpha t)f_2(x+t)\frac{1}{t}dt .$$

In two papers from 1997 and 1999, Lacey and Thiele showed that the bilinear Hilbert transform $BHT_{\alpha}$ maps $L^p \times L^q$ into $L^r$ when $\frac{1}{p} + \frac{1}{q} = \frac{1}{r}$, $1<p,q\leq \infty$ and $\frac{2}{3} <r<\infty$ with a bound depending on $\alpha$ \cite{LT1, LT2}. Uniform boundedness of these $L^p$ estimates was resolved later \cite{GL1, T1}. Note that $r$ only goes down to $\frac{2}{3}$, not $\frac{1}{2}$ as one would expect from H\"{o}lder type estimates. It is still an open problem whether $r$ can be pushed all the way down to $\frac{1}{2}$.

In a similar fashion then one can rewrite the second Calder\'{o}n commutator with two averages. Dropping both averages motivates the definition of the trilinear Hilbert transform.

$$ THT_{\vec{\alpha}}(f_1, f_2, f_3)(x) = p.v. \int\limits_{\mathbb{R}}f_1(x+\alpha_1 t)f_2(x+\alpha_2 t)f_3(x+t)\frac{1}{t}dt $$

\noindent In contrast to the bilinear Hilbert transform then no $L^p$ estimates are known for the trilinear Hilbert transform.

In this paper we will study a trilinear operator motivated by $C_{A}^{(2)}$ in a similar fashion as $THT_{\vec{\alpha}}$, except we drop one average, not two. Define

\begin{equation}\label{operator}
T_{\beta}(f_1, f_2, f_3)(x) := p.v. \int\limits_{\mathbb{R}}\left(\int_{0}^{1}f_1(x+\alpha t) d\alpha\right)f_2(x+\beta t)f_3(x+t)\frac{1}{t}dt .
\end{equation}

\subsection{Known estimates}

Benyi, Demeter, Nahmod, Thiele, Torres and Villarroya obtained a modulation invariant bilinear $T(1)$ theorem \cite{BDNTTV1}. If one fixes $f_1\in L^{\infty}(\mathbb{R})$ and looks at the bilinear operator

$$ (f_2, f_3) \mapsto p.v. \int\limits_{\mathbb{R}}\frac{\left(\int_{0}^{1}f_1(x+\alpha t) d\alpha\right)}{t}f_2(x+\beta t)f_3(x+t)dt , $$

\noindent one can apply their theorem to obtain the following $L^p$ estimates for $T_{\beta}$

$$ T_{\beta} : L^{\infty}\times L^{p_1}\times L^{p_2} \rightarrow L^p $$

\noindent for $\beta \notin \lbrace 0,1 \rbrace$ if $\displaystyle{\frac{1}{p_1} + \frac{1}{p_2} = \frac{1}{p}}$, $1<p_1, p_2 \leq \infty$ and $\frac{2}{3} < p < \infty$. These are the only known $L^p$ estimates for $T_{\beta}$.

\subsection{Result}

The main theorem of this paper establishes the following wide range of $L^p$ estimates for $T_{\beta}$.

\begin{theorem}\label{original_thm}
Let $\beta \notin \lbrace 0,1 \rbrace$, $1<p_1, p_2, p_3 \leq \infty$,
$$\frac{1}{2}< p:= \frac{p_1p_2 p_3}{p_1p_2 + p_1p_3 + p_2p_3} < \infty  \quad \text{ and } \quad \frac{2}{3}< \frac{p_2 p_3}{p_2 + p_3} \leq \infty.$$

\noindent Then there exists a constant $C_{\beta,p_1,p_2,p_3}$ such that
$$ \|T_{\beta}(f_1, f_2, f_3)\|_{p} \leq C_{\beta,p_1,p_2,p_3} \|f_1\|_{p_1}\|f_2\|_{p_2}\|f_3\|_{p_3} $$
for all $f_1$, $f_2$ and $f_3$ in $\mathcal{S}(\mathbb{R})$.
\end{theorem}

\noindent The theorem recovers all known $L^p$ estimates for the operator. Known $L^p$ estimates for both the bilinear Hilbert transform and for Calder\'{o}n's first commutator follow as a corollary.

Compared to the theorem on the bilinear Hilbert transform, this theorem has an extra condition.

$$ \frac{2}{3}< \frac{p_2 p_3}{p_2 + p_3} \leq \infty $$

\noindent This condition implies that we have not improved the previously known $L^p$ estimates for the bilinear Hilbert transform. We also require the condition $\frac{1}{2}< p$, which is not the largest possible range of $L^p$ estimates expected. Based on the known estimates for the bilinear Hilbert transform one would expect to be able to go all the way down to $\frac{2}{5}$. This remains an open problem. 

Note that if $\beta =0,1$ then we obtain trilinear operators that only involve multiplication of functions and the first Calder\'{o}n commutator. The $L^p$-bounds of these operators are easy to determine.

\subsection{Approach}

The standard way of understanding the boundedness of the Calder\'{o}n commutators is to use the $T(1)$ theorem. In order to use such an approach on $T_{\beta}$ we would need some sort of a trilinear $T(1)$ theorem. Despite the existence of some multilinear $T(1)$ theorems \cite{CJ1,GT1} then there is no such appropriate theorem for $T_{\beta}$. The other canonical way of trying to understand $T_{\beta}$ would be to establish uniform $L^p$ estimates on the trilinear Hilbert transform. Since no $L^p$ estimates exist, uniform estimates are out of reach. The obvious approaches to find $L^p$ estimates fail so we need some novel ideas.

On the Fourier side it is equivalent to show $L^p$ estimates for an operator $T_{\beta}$ given by

\begin{multline}
T_{\beta}(f_1, f_2, f_3)(x) = \int\limits_{\mathbb{R}^3} \! \left[\int_{0}^{1}\mathsf{sgn}(\alpha\xi_1 \! + \! \beta\xi_2 \! + \! \xi_3)d\alpha\right]\widehat{f_1}(\xi_1)\widehat{f_2}(\xi_2)\widehat{f_3}(\xi_3) \\
e^{2\pi i (\xi_1+\xi_2+\xi_3)x}d\xi_1 d\xi_2 d\xi_3 .
\end{multline}

\noindent where $\mathsf{sgn}$ is the usual sign function. The symbol $\int_{0}^{1}\mathsf{sgn}(\alpha\xi_1 \! + \! \beta\xi_2 \! + \! \xi_3)d\alpha$ has a singularity around the line $\xi_1=0$, $\beta\xi_2 + \xi_3 = 0$ in the sense that it is discontinuous. This is similar to the bilinear Hilbert transform.  Unlike standard symbols, which are assumed to be smooth outside the set where they are singular, this symbol is continuous but not differentiable on the planes $\xi_1 + \beta\xi_2 + \xi_3 = 0$ and $\xi_1 = 0$ away from the previous line. We approach the symbol as a rough non-standard symbol and use techniques in the spirit of the bilinear Hilbert transform. An important ingredient in that approach are new proofs of the $L^p$ estimates for the Calder\'{o}n commutators by Muscalu \cite{M1}. The techniques and notation are also heavily inspired by Muscalu, Tao and Thiele \cite{MTT1,MTT2}.

There exist theorems that give immediate $L^p$ estimates for operators with standard symbols where the dimension of the singularity is strictly less than half the dimension of the frequency space of the form associated to the operator \cite{MTT4}. Even if our symbol had been standard outside the line then those kind of theorems would not have been applicable because the line is degenerate.

\subsection{Acknowledgements}

The author would like to thank his thesis adviser, Camil Muscalu, for his guidance and many helpful conversations about this problem.

\section{Notation}

We use $A\lesssim B$ to denote the statement that $A\leq CB$ for some large constant $C$ and $A\ll B$ to denote the statement that $A \leq C^{-1}B$ for some large constant $C$. Our constants $C$ shall always be independent of the tiles $\vec{P}$.

Given any interval $I$, let $|I|$ denote the Lebesgue measure of $I$ and let $cI$ denote the interval with the same center as $I$ but $c$ times the side-length. Also define the approximate cutoff function $\tilde{\chi}_{I}$ by

$$ \tilde{\chi}_{I}(x) := (1+(\frac{|x-x_I|}{|I|})^2)^{-1/2} $$

\noindent where $x_I$ is the center of $I$.

Define $\langle n \rangle := 2 + |n|$ for $n\in \mathbb{Z}$.

\section{Symbol}

The meaning of \eqref{operator} is

\begin{equation}\label{op_limit}
\lim\limits_{\epsilon\rightarrow 0^{+}} \int\limits_{|t|>\epsilon}\left(\int_{0}^{1}f_1(x+\alpha t) d\alpha\right)f_2(x+\beta t)f_3(x+t)\frac{1}{t}dt
\end{equation}

\noindent where the limit exists. Assume $f_1$, $f_2$ and $f_3$ are Schwartz functions on $\mathbb{R}$. We will show that \eqref{op_limit} exists in that case and we will rewrite it in a convenient way.

Write \eqref{op_limit} as

\begin{multline*}
\lim\limits_{\substack{\epsilon\rightarrow 0^{+} \\ N \rightarrow \infty}} \quad \int\limits_{\epsilon < |t| < N}\left[\int_{0}^{1} \int_{\mathbb{R}}\widehat{f_1}(\xi_1)e^{2\pi i \xi_1(x+\alpha t)}d\xi_1 \ d\alpha\right] \\
\int_{\mathbb{R}}\widehat{f_2}(\xi_2)e^{2\pi i \xi_2(x+\beta t)}d\xi_2 \int_{\mathbb{R}}\widehat{f_3}(\xi_3)e^{2\pi i \xi_3(x+ t)}d\xi_3 \ \frac{1}{t}dt
\end{multline*}

\noindent which is equal to

\begin{multline*}
\lim\limits_{\substack{\epsilon\rightarrow 0^{+} \\ N \rightarrow \infty}} \quad \int\limits_{\epsilon < |t| < N}\int_{\mathbb{R}^3}\left[\int_{0}^{1} \frac{1}{t} e^{-2\pi i (-\alpha \xi_1 - \beta \xi_2 - \xi_3)} d\alpha\right] \widehat{f_1}(\xi_1)\widehat{f_2}(\xi_2)\widehat{f_3}(\xi_3)e^{2\pi i x(\xi_1+\xi_2+\xi_3)} d\xi_1 d\xi_2 d\xi_3 dt
\end{multline*}

\noindent The function being integrated, viewed as depending on $\xi_1$, $\xi_2$, $\xi_3$ and $t$ is clearly absolutely integrable on $\mathbb{R}^4$ and by applying Fubini's theorem together with dominated convergence we see that the formula becomes equivalent to

\begin{multline}\label{four}
\int\limits_{\mathbb{R}^3} \! \left[\int_{0}^{1}\mathsf{sgn}(-\alpha\xi_1 \! - \! \beta\xi_2 \! - \! \xi_3)d\alpha\right]\widehat{f_1}(\xi_1)\widehat{f_2}(\xi_2)\widehat{f_3}(\xi_3) e^{2\pi i (\xi_1+\xi_2+\xi_3)x}d\xi_1 d\xi_2 d\xi_3
\end{multline}

\noindent which clearly exists since $\widehat{f_1}$, $\widehat{f_2}$ and $\widehat{f_3}$ are also Schwartz functions.

A product of three functions satisfies a H\"{o}lder type inequality as we obtain in Theorem \ref{original_thm}. Since the product can be written as

\begin{equation}\label{mult}
\int\limits_{\mathbb{R}^3} \! \widehat{f_1}(\xi_1)\widehat{f_2}(\xi_2)\widehat{f_3}(\xi_3) e^{2\pi i (\xi_1+\xi_2+\xi_3)x}d\xi_1 d\xi_2 d\xi_3
\end{equation}

\noindent and using $\mathsf{sgn}(-x) = - \mathsf{sgn}(x)$ it becomes clear by subtracting \eqref{four} from \eqref{mult} that it is enough to consider $L^p$ estimates for

\begin{multline}\label{real_op}
\tilde{T}_{\beta}(f_1, f_2, f_3)(x) := \int\limits_{\mathbb{R}^3} \! \left[\int_{0}^{1}1_{\mathbb{R}_{+}}(\alpha\xi_1 \! + \! \beta\xi_2 \! + \! \xi_3)d\alpha\right]\widehat{f_1}(\xi_1)\widehat{f_2}(\xi_2)\widehat{f_3}(\xi_3) \\
e^{2\pi i (\xi_1+\xi_2+\xi_3)x}d\xi_1 d\xi_2 d\xi_3 .
\end{multline}

\noindent where $1_{\mathbb{R}_{+}}$ is the characteristic function for the positive real axis.

Similar to what was mentioned in the introduction then the symbol

\begin{equation*}
\int_{0}^{1}1_{\mathbb{R}_{+}}(\alpha\xi_1 \! + \! \beta\xi_2 \! + \! \xi_3)d\alpha
\end{equation*}

\noindent is not continuous around the line $\xi_1=0$, $\beta\xi_2 + \xi_3 = 0$, continuous but not differentiable around the planes $\xi_1 + \beta\xi_2 + \xi_3 = 0$ and $\beta\xi_2 + \xi_3 = 0$, away from the previous line, but smooth everywhere else. It is tempting to view the symbol as a trilinear symbol of the variables $\xi_1$, $\xi_2$, $\xi_3$. That would however result in a problem of the same difficulty as the trilinear Hilbert transform. We choose thus instead to view it as a non-standard bilinear symbol of the variables $\xi_1$ and $\beta \xi_2 + \xi_3$.

\section{Discretization}

We will now come up with a "discretized" variant of the "continuous" form associated to \eqref{real_op}. We start by reviewing some standard definitions and comments \cite{MTT2}.

\begin{definition}
Let $n \geq 1$ and $\sigma \in \lbrace 0, \frac{1}{3}, \frac{2}{3} \rbrace^{n}$. We define the shifted $n$-dyadic mesh $D = D_{\sigma}^{n}$ to be the collection of cubes of the form

\begin{equation*}
D_{\sigma}^{n} := \lbrace 2^j (k+(0,1)^n + (-1)^j \sigma )| j\in\mathbb{Z}, \quad k\in\mathbb{Z}^n \rbrace 
\end{equation*}

\noindent We define a shifted dyadic cube to be any member of a shifted $n$-dyadic mesh.
\end{definition}

Observe that for every cube $Q$, there exists a shifted dyadic cube $Q'$ such that $Q \subseteq \frac{7}{10}Q'$ and $|Q'| \sim |Q|$; this is best seen by first verifying the $n=1$ case.

\begin{definition}
A subset $D'$ of a shifted $n$-dyadic grid $D$ is called sparse, if for any two cubes $Q$, $Q'$ in $D$ with $Q \neq Q'$ we have $|Q| < |Q'|$ implies $|10^9 Q| < |Q'|$ and $|Q| = |Q'|$ implies $10^9 Q \cap 10^9 Q' = \emptyset$.
\end{definition}

Observe that any subset of a shifted $n$-dyadic grid (with $n\leq 4$ say), can be split into $O(1)$ sparse subsets.

\begin{definition}
Let $\sigma = (\sigma_1, \sigma_2, \sigma_3, \sigma_4) \in \lbrace 0, \frac{1}{3}, \frac{2}{3} \rbrace^4$, and let $1 \leq i \leq 4$. An $i$-tile with shift $\sigma_i$ is a rectangle $P = I_{P}\times \omega_P$ with area $1$ and with $I_P \in D_{0}^{1}$, $\omega_P \in D_{\sigma_i}^{1}$. A quadtile with shift $\sigma$ is a $4$-tuble $\vec{P} = (P_1, P_2, P_3, P_4)$ such that each $P_i$ is an $i$-tile with shift $\sigma_i$, and the $I_{P_i} = I_{\vec{P}}$ are independent of $i$. The frequency cube $Q_{\vec{P}}$ of a quadtile is defined to be $\Pi_{i=1}^{4}\omega_{P_i}$
\end{definition}

We sometimes refer to $i$-tiles with shift $\sigma$ just as $i$-tiles, or even as tiles, if the parameters $\sigma$, $i$ are unimportant.

\begin{definition}
A set $\vec{\mathbf{P}}$ of quadtiles is called sparse, if all quadtiles in $\vec{\mathbf{P}}$ have the same shift and the set $\lbrace Q_{\vec{P}}:\vec{P}\in\vec{\mathbf{P}}\rbrace$ is sparse.
\end{definition}

Again, any set of quadtiles can be split into $O(1)$ sparse subsets.

\begin{definition}
Let $P$ and $P'$ be tiles. We write $P' < P$ if $I_{P'} \subsetneq I_P$ and $5\omega_P \subseteq 5\omega_{P'}$, and $P' \leq P$ if $P' < P$ or $P' = P$. We write $P' \lesssim P$ if $I_{P'} \subseteq I_P$ and $10^7\omega_P \subseteq 10^7 \omega_{P'}$. We write $P' \lesssim' P$ if $P' \lesssim P$ and $P' \nleq P$.
\end{definition}

This ordering by Muscalu, Tao and Thiele \cite{MTT2} is in the spirit of that in Fefferman \cite{F1} or Lacey and Thiele \cite{LT1, LT2}. The main difference from the previous orderings is that $P'$ and $P$ do not quite have to intersect which turns out to be convenient for technical purposes.

\begin{definition}
Let $P$ be a tile. An $L^p$ normalized wave packet on $P$, $1 \leq p < \infty$, is a function $\phi_P$ which has Fourier support in $\frac{9}{10}\omega_{P}$ and obeys the estimates

\begin{equation*}
|\phi_P(x)| \lesssim |I_P|^{-1/p}\tilde{\chi}_{I}(x)^M
\end{equation*}

\noindent for all $M>0$, with the implicit constant depending on $M$.
\end{definition}

\noindent Heuristically, $\phi_P$ is $L^p$-normalized and is supported in $P$.

Now that we have the tools from Muscalu, Tao and Thiele \cite{MTT2} then let us start decomposing. We start with two standard Littlewood-Paley decompositions and write

\begin{equation*}
1_{\mathbb{R}}(\xi_1) = \sum\limits_{k_1}\widehat{\Psi_{k_1}}(\xi_1)
\end{equation*}

\noindent and

\begin{equation*}
1_{\mathbb{R}}(\beta\xi_2+\xi_3) = \sum\limits_{k_2}\widehat{\Psi_{k_2}}(\beta\xi_2 + \xi_3)
\end{equation*}

\noindent where as usual, $\widehat{\Psi_{k_1}}(\xi_1)$ and $\widehat{\Psi_{k_2}}(\beta\xi_2 + \xi_3)$ are bumps supported in the regions $|\xi_1| \sim 2^{k_1}$ and $|\beta\xi_2+\xi_3|\sim 2^{k_2}$ respectively. In particular we get

\begin{equation}\label{littlewood_paley}
1_{\mathbb{R}}(\xi_1 , \beta\xi_2 + \xi_3) = \sum\limits_{k_1, k_2} \widehat{\Psi_{k_1}}(\xi_1)\widehat{\Psi_{k_2}}(\beta\xi_2 + \xi_3)
\end{equation}

\noindent By splitting \eqref{littlewood_paley} over the regions where $k_1 \ll k_2$, $k_2 \ll k_1$ and $k_1 \sim k_2$ we obtain the decomposition

\begin{equation}\label{case1}
1_{\mathbb{R}}(\xi_1 , \beta\xi_2 + \xi_3) = \sum\limits_{k}\widehat{\Psi_{k}}(\xi_1)\widehat{\Phi_{k}}(\beta\xi_2 + \xi_3) \quad +
\end{equation}
\begin{equation}\label{case2}
\sum\limits_{k}\widehat{\Phi_{k}}(\xi_1)\widehat{\Psi_{k}}(\beta\xi_2 + \xi_3) \quad +
\end{equation}
\begin{equation}\label{case3}
\sum\limits_{k_1 \sim k_2}\widehat{\Psi_{k_1}}(\xi_1)\widehat{\Psi_{k_2}}(\beta\xi_2 + \xi_3).
\end{equation}

\noindent where $\widehat{\Phi_{k}}$ is a bump supported on an interval, symmetric with respect to the origin of length $\sim 2^{k}$.

Note that $\widehat{\Phi_{k}}(\beta\xi_2 + \xi_3)$ is supported in $\mathbb{R}^2$ on a strip around the line $\beta\xi_2 + \xi_3 = 0$ of width $\sim 2^k$. We can cover that strip with shifted dyadic cubes with side-length $\sim 2^k$. Similarly then $\widehat{\Psi_{k}}(\beta\xi_2+\xi_3)$ is supported in $\mathbb{R}^2$ on two strips of width $\sim 2^k$ but this time away from $\beta\xi_2 + \xi_3=0$. Again we can cover those strips with shifted dyadic cubes of a similar scale.

Thus we come up with a decomposition

\begin{equation}\label{decomp1}
a(\xi_1, \xi_2, \xi_3) = \sum\limits_{\vec{Q} \in \vec{\mathbf{Q}}}\widehat{\phi_{Q_1, 1}}(\xi_1)\widehat{\phi_{Q_2, 2}}(\xi_2)\widehat{\phi_{Q_3, 3}}(\xi_3)
\end{equation}

for each of the three cases \eqref{case1}, \eqref{case2}, \eqref{case3} such that

\begin{equation*}
\frac{1}{10} < a(\xi_1,\xi_2,\xi_3) < 10.
\end{equation*}

\noindent Here $\phi_{Q_i, i}$ is an $L^1$ normalized wave packet on a tile $I_{\vec{Q}}\times Q_i$ for $i=1,2,3$, where $Q_i$ is a shifted dyadic interval that depends on the decomposition in each of the three cases and $I_{\vec{Q}}$ is a dyadic interval such that $|I_{\vec{Q}}| \sim |Q_i|^{-1}$ for $i=1,2,3$.

Since $\xi_1 \in \frac{9}{10}Q_1$, $\xi_2 \in \frac{9}{10}Q_2$ and $\xi_3 \in \frac{9}{10}Q_3$ it follows that $\xi_1 + \xi_2 + \xi_3 \in  \frac{9}{10}Q_1 +  \frac{9}{10}Q_2 +  \frac{9}{10}Q_3$ and as a consequence one can find a shifted dyadic interval $Q_4$ with the property that $\frac{9}{10}Q_1 +  \frac{9}{10}Q_2 +  \frac{9}{10}Q_3 \subseteq -\frac{7}{10}Q_4$ and also satisfying $|Q_1| = |Q_2| = |Q_3| \sim |Q_4|$. In particular there exists an $L^1$ normalized wave packet $\phi_{Q_4, 4}$ adapted to $I\times Q_4$ such that $\widehat{\phi_{Q_4, 4}} \equiv 1$ on $-\frac{9}{10}Q_1 -  \frac{9}{10}Q_2 - \frac{9}{10}Q_3$.

Thus \eqref{decomp1} can be written as

\begin{equation}\label{decomp2}
a(\xi_1, \xi_2, \xi_3) = \sum\limits_{\vec{Q} \in \vec{\mathbf{Q}}}\widehat{\phi_{Q_1, 1}}(\xi_1)\widehat{\phi_{Q_2, 2}}(\xi_2)\widehat{\phi_{Q_3, 3}}(\xi_3)\widehat{\phi_{Q_4, 4}}(-\xi_1 - \xi_2 - \xi_3)
\end{equation}

\noindent where this time $\vec{\mathbf{Q}}$ is a collection of shifted dyadic quasi-cubes in $\mathbb{R}^4$. Modulo a finite refinement we can assume that a sum of the type

\begin{equation}\label{decomp3}
\sum\limits_{\vec{Q} \in \vec{\mathbf{Q}}}\widehat{\phi_{Q_1, 1}}(\xi_1)\widehat{\phi_{Q_2, 2}}(\xi_2)\widehat{\phi_{Q_3, 3}}(\xi_3)\widehat{\phi_{Q_4, 4}}(-\xi_1 - \xi_2 - \xi_3)
\end{equation}

\noindent runs over a sparse collection of tiles $\vec{\mathbf{Q}}$. In such a sparse collection, then for every $Q\in\vec{\mathbf{Q}}$ there exists a unique shifted cube $\tilde{Q}$ in $\mathbb{R}^4$ such that $Q \subseteq \frac{7}{10}\tilde{Q}$ and with the diameter of $Q$ similar to the diameter of $\tilde{Q}$. This allows us to assume that a sum of the type \eqref{decomp3} runs over a sparse collections of shifted dyadic cubes such that $|Q_1| \sim |Q_2| \sim |Q_3| \sim |Q_4|$. Let $|\vec{Q}| \sim |Q_i|$, $i=1,2,3,4$, be the scale of the dyadic cube.

Further we know that in all three cases \eqref{case1}, \eqref{case2} and \eqref{case3} then the scale $|\vec{Q}|$ fixes the location of the tile $Q_1$. Also in the case \eqref{case1} where we are close to the line $\beta\xi_2 + \xi_3 = 0$ then the tiles $Q_2$ and $Q_3$ can be made to overlap while in the second two cases \eqref{case2}, \eqref{case3}, when we are away from the line $\beta\xi_2 + \xi_3 = 0$ then $Q_2$ and $Q_3$ can be made to be a couple of units of length $|\vec{Q}|$ away from another so they don't overlap.

We will now study the quadlinear form associated to \eqref{real_op}.

\begin{multline}\label{form}
\int_{\mathbb{R}}\tilde{T}_{\beta}(f_1, f_2, f_3)(x)f_4(x)dx \\
= \int\limits_{\xi_1 + \xi_2 + \xi_3 + \xi_4 = 0} \left[\int_{0}^{1}1_{\mathbb{R}_{+}}(\alpha\xi_1 \! + \! \beta\xi_2 \! + \! \xi_3)d\alpha\right]\widehat{f_1}(\xi_1)\widehat{f_2}(\xi_2)\widehat{f_3}(\xi_3)\widehat{f_4}(\xi_4)d\xi_1 d\xi_2 d\xi_3 d\xi_4 \\
= \sum\limits_{\vec{Q} \in \vec{\mathbf{Q}}}\ \int\limits_{\xi_1 + \xi_2 + \xi_3 + \xi_4 = 0} \frac{ \left[\int_{0}^{1}1_{\mathbb{R}_{+}}(\alpha\xi_1 \! + \! \beta\xi_2 \! + \! \xi_3)d\alpha\right]}{a(\xi_1, \xi_2, \xi_3)}\widehat{\phi_{Q_1, 1}}(\xi_1)\widehat{\phi_{Q_2, 2}}(\xi_2)\widehat{\phi_{Q_3, 3}}(\xi_3)\widehat{\phi_{Q_4, 4}}(-\xi_1 - \xi_2 - \xi_3) \\
\widehat{f_1}(\xi_1)\widehat{f_2}(\xi_2)\widehat{f_3}(\xi_3)\widehat{f_4}(\xi_4)d\xi_1 d\xi_2 d\xi_3 d\xi_4 \\
= \sum\limits_{\vec{Q} \in \vec{\mathbf{Q}}}\ \int\limits_{\xi_1 + \xi_2 + \xi_3 + \xi_4 = 0} \frac{ \left[\int_{0}^{1}1_{\mathbb{R}_{+}}(\alpha\xi_1 \! + \! \beta\xi_2 \! + \! \xi_3)d\alpha\right]}{a(\xi_1, \xi_2, \xi_3)}\widehat{\phi_{Q_1, 1}}(\xi_1)\widehat{\phi_{Q_2, 2}}(\xi_2)\widehat{\phi_{Q_3, 3}}(\xi_3)\widehat{\phi_{Q_4, 4}}(\xi_4) \\
\widehat{f_1}(\xi_1)\widehat{f_2}(\xi_2)\widehat{f_3}(\xi_3)\widehat{f_4}(\xi_4)d\xi_1 d\xi_2 d\xi_3 d\xi_4 \\
\end{multline}

\noindent We can write

\begin{equation*}
\frac{ \left[\int_{0}^{1}1_{\mathbb{R}_{+}}(\alpha\xi_1 \! + \! \beta\xi_2 \! + \! \xi_3)d\alpha\right]}{a(\xi_1, \xi_2, \xi_3)}\widehat{\phi_{Q_1, 1}}(\xi_1)\widehat{\phi_{Q_2, 2}}(\xi_2)\widehat{\phi_{Q_3, 3}}(\xi_3)
\end{equation*}

\noindent as

\begin{equation*}
\frac{ \left[\int_{0}^{1}1_{\mathbb{R}_{+}}(\alpha\xi_1 \! + \! \beta\xi_2 \! + \! \xi_3)d\alpha\right]}{a(\xi_1, \xi_2, \xi_3)}\widehat{\tilde{\phi}_{Q_1, 1}}(\xi_1)\widehat{\tilde{\phi}_{Q_2, 2}}(\xi_2)\widehat{\tilde{\phi}_{Q_3, 3}}(\xi_3) 
\widehat{\phi_{Q_1, 1}}(\xi_1)\widehat{\phi_{Q_2, 2}}(\xi_2)\widehat{\phi_{Q_3, 3}}(\xi_3)
\end{equation*}

\noindent where $\widehat{\tilde{\phi}_{Q_1, 1}} \otimes \widehat{\tilde{\phi}_{Q_2, 2}} \otimes \widehat{\tilde{\phi}_{Q_3, 3}}$ is identically equal to $1$ on the support of $\widehat{\phi_{Q_1, 1}} \otimes \widehat{\phi_{Q_2, 2}} \otimes \widehat{\phi_{Q_3, 3}}$.

Now split

\begin{equation*}
\frac{ \left[\int_{0}^{1}1_{\mathbb{R}_{+}}(\alpha\xi_1 \! + \! \beta\xi_2 \! + \! \xi_3)d\alpha\right]}{a(\xi_1, \xi_2, \xi_3)}\widehat{\tilde{\phi}_{Q_1, 1}}(\xi_1)\widehat{\tilde{\phi}_{Q_2, 2}}(\xi_2)\widehat{\tilde{\phi}_{Q_3, 3}}(\xi_3)
\end{equation*}

\noindent as a Fourier series

\begin{equation*}
\sum\limits_{n_1, n_2, n_3}C_{n_1, n_2, n_3}^{\vec{Q}}e^{2\pi i \frac{n_1}{|\vec{Q}|}\xi_1}e^{2\pi i \frac{n_2}{|\vec{Q}|}\xi_2}e^{2\pi i \frac{n_3}{|\vec{Q}|}\xi_3}.
\end{equation*}

\noindent The coefficient $C_{n_1, n_2, n_3}^{\vec{Q}}$ is given by

\begin{multline}\label{fourier_coeff}
C_{n_1, n_2, n_3}^{\vec{Q}} = \frac{1}{|\vec{Q}|^4} \int_{\mathbb{R}^3}\frac{ \left[\int_{0}^{1}1_{\mathbb{R}_{+}}(\alpha\xi_1 \! + \! \beta\xi_2 \! + \! \xi_3)d\alpha\right]}{a(\xi_1, \xi_2, \xi_3)}\widehat{\tilde{\phi}_{Q_1, 1}}(\xi_1)\widehat{\tilde{\phi}_{Q_2, 2}}(\xi_2)\widehat{\tilde{\phi}_{Q_3, 3}}(\xi_3) \\
e^{-2\pi i \frac{n_1}{|\vec{Q}|}\xi_1}e^{-2\pi i \frac{n_2}{|\vec{Q}|}\xi_2}e^{-2\pi i \frac{n_3}{|\vec{Q}|}\xi_3} d\xi_1 d\xi_2 d\xi_3
\end{multline}

\begin{lemma}

\begin{equation*}
|C_{n_1, n_2, n_3}^{\vec{Q}}| \lesssim C(n_1, n_2, n_3)
\end{equation*}

\noindent where the implicit constant does not depend on $\vec{Q}$.

\end{lemma}

This lemma is a consequence of lemma \ref{decay} that we prove in section 6. The main point for now is that the Fourier coefficient is bounded uniformly independently of the dyadic cube $\vec{Q}$.

We can now majorize \eqref{form} by

\begin{multline*}
\sum\limits_{n_1, n_2, n_3}C(n_1, n_2, n_3)\sum\limits_{\vec{Q} \in \vec{\mathbf{Q}}}|\int\limits_{\xi_1 + \xi_2 + \xi_3 + \xi_4 = 0}\widehat{\phi_{Q_1, 1}}(\xi_1)\widehat{\phi_{Q_2, 2}}(\xi_2)\widehat{\phi_{Q_3, 3}}(\xi_3)\widehat{\phi_{Q_4, 4}}(\xi_4)\\
\widehat{f_1}(\xi_1)\widehat{f_2}(\xi_2)\widehat{f_3}(\xi_3)\widehat{f_4}(\xi_4)e^{-2\pi i \frac{n_1}{|\vec{Q}|}\xi_1}e^{-2\pi i \frac{n_2}{|\vec{Q}|}\xi_2}e^{-2\pi i \frac{n_3}{|\vec{Q}|}\xi_3}d\xi_1 d\xi_2 d\xi_3 d\xi_4| \\
= \sum\limits_{n_1, n_2, n_3}C(n_1, n_2, n_3)\sum\limits_{\vec{Q} \in \vec{\mathbf{Q}}}|\int\limits_{\mathbb{R}^5}\widehat{f_1 * \phi_{Q_1, 1}}(\xi_1)\widehat{f_2 * \phi_{Q_2, 2}}(\xi_2)\widehat{f_3 * \phi_{Q_3, 3}}(\xi_3)\widehat{f_4 * \phi_{Q_4, 4}}(\xi_4) \\
e^{2\pi i (\xi_1+\xi_2+\xi_3+\xi_4)x}d\xi_1 d\xi_2 d\xi_3 d\xi_4 dx| \\
=  \sum\limits_{n_1, n_2, n_3}C(n_1, n_2, n_3)\sum\limits_{\vec{Q} \in \vec{\mathbf{Q}}}|\int_{\mathbb{R}}(f_1 * \phi_{Q_1, 1}^{n_1})(x)(f_2 * \phi_{Q_2, 2}^{n_2})(x)(f_3 * \phi_{Q_3, 3}^{n_3})(x)(f_4 * \phi_{Q_4, 4})(x)dx|
\end{multline*}

\noindent Here the meaning of $\phi_{Q_i, i}^{n_i}$ is that if $\phi_{Q_i, i}$ was an $L^1$ normalized wave packet on $I_{\vec{Q}}\times Q_i$ then $\phi_{Q_i, i}^{n_i}$ is an $L^1$ normalized wave packet on $I_{\vec{Q}}^{n_i}\times Q_i$ where $I_{\vec{Q}}^{n_i}$ is a dyadic interval sitting $n_i$ units of length $|I_{\vec{Q}}|$ away from $I_{\vec{Q}}$.

Split $\vec{\mathbf{Q}} = \bigcup_{k\in\mathbb{Z}}\vec{\mathbf{Q}}_k$ where $\vec{\mathbf{Q}}_k$ has cubes $\vec{Q}$ of scale $|\vec{Q}| = 2^k$ and thus $|I_{\vec{Q}}| = 2^{-k}$.

\begin{multline}\label{discr}
\sum\limits_{\vec{Q} \in \vec{\mathbf{Q}}}|\int_{\mathbb{R}}(f_1 * \phi_{Q_1, 1}^{n_1})(x)(f_2 * \phi_{Q_2, 2}^{n_2})(x)(f_3 * \phi_{Q_3, 3}^{n_3})(x)(f_4 * \phi_{Q_4, 4})(x)dx| \\
= \sum\limits_{k\in\mathbb{Z}}\sum\limits_{\vec{Q} \in \vec{\mathbf{Q}}_k}|2^{-k}\int_{\mathbb{R}}(f_1 * \phi_{Q_1, 1}^{n_1})(2^{-k}y)(f_2 * \phi_{Q_2, 2}^{n_2})(2^{-k}y)(f_3 * \phi_{Q_3, 3}^{n_3})(2^{-k}y)(f_4 * \phi_{Q_4, 4})(2^{-k}y)dy| \\
= \sum\limits_{k\in\mathbb{Z}}\sum\limits_{\vec{Q} \in \vec{\mathbf{Q}}_k}\Bigl\lvert|I_{\vec{Q}}|\int_{0}^{1}\sum\limits_{m\in\mathbb{Z}}(f_1 * \phi_{Q_1, 1}^{n_1})(2^{-k}m+2^{-k}\gamma)(f_2 * \phi_{Q_2, 2}^{n_2})(2^{-k}m + 2^{-k}\gamma)\\
(f_3 * \phi_{Q_3, 3}^{n_3})(2^{-k}m + 2^{-k}\gamma)(f_4 * \phi_{Q_4, 4})(2^{-k}m + 2^{-k}\gamma)d\gamma\Bigr\rvert
\end{multline}

\noindent Now observe that for $i=1,2,3,4$ (where we take $n_4 = 0$)

\begin{align*}
(f_i * \phi_{Q_i, i}^{n_i})(2^{-k}m + 2^{-k}\gamma) &= \int_{\mathbb{R}}f_i(z)\phi_{Q_i, i}^{n_i}(2^{-k}m + 2^{-k}\gamma - z)dz \\
&= \frac{1}{|I_{\vec{Q}}|^{1/2}}\int_{\mathbb{R}}f_i(z)|I_{\vec{Q}}|^{1/2}\phi_{Q_i, i}^{n_i}(2^{-k}m + 2^{-k}\gamma - z)dz \\
&= \frac{1}{|I_{\vec{Q}}|^{1/2}}\langle f_i , \tilde{\phi}_{Q_i, i, m, \gamma}^{n_i} \rangle
\end{align*}

\noindent where $\tilde{\phi}_{Q_i, i, m, \gamma}^{n_i}$ is a wave packet translated from $\phi_{Q_i, i}^{n_i}$ by $m$ steps in time and then additionally shifted by $\gamma$ steps. Note that $\tilde{\phi}_{Q_i, i, m, \gamma}^{n_i}$ is an $L^2$ normalized wave packet since $\phi_{Q_i, i}^{n_i}$ was $L^1$ normalized. Now \eqref{discr} becomes

\begin{multline*}
\int_{0}^{1}\sum\limits_{\vec{Q} \in \vec{\mathbf{Q}}}\sum\limits_{m\in\mathbb{Z}}|I_{\vec{Q}}|\frac{1}{|I_{\vec{Q}}|^{1/2}}\langle f_1 , \tilde{\phi}_{Q_1, 1, m, \gamma}^{n_1} \rangle \frac{1}{|I_{\vec{Q}}|^{1/2}}\langle f_2 , \tilde{\phi}_{Q_2, 2, m, \gamma}^{n_2} \rangle \frac{1}{|I_{\vec{Q}}|^{1/2}}\langle f_3 , \tilde{\phi}_{Q_3, 3, m, \gamma}^{n_3} \rangle \\
\frac{1}{|I_{\vec{Q}}|^{1/2}}\langle f_4 , \tilde{\phi}_{Q_4, 4, m, \gamma} \rangle d\gamma \\
= \int_{0}^{1}\sum\limits_{\vec{P} \in \vec{\mathbf{P}}}\frac{1}{|I_{\vec{P}}|}\langle f_1 , \tilde{\phi}_{P_1^{n_1}, 1, \gamma} \rangle \langle f_2 , \tilde{\phi}_{P_2^{n_2}, 2, \gamma} \rangle \langle f_3 , \tilde{\phi}_{P_3^{n_3}, 3, \gamma} \rangle \langle f_4 , \tilde{\phi}_{P_4, 4, \gamma} \rangle d\gamma
\end{multline*}

\noindent where $P_i^{n_i}$ denotes the tile $I_{P_i}^{n_i+m} \times Q_i$ where $I_{P_i}^{n_i+m}$ is a dyadic interval such that $|I_{P_i}^{n_i+m}| \sim |Q_i|^{-1}$ for $i=1,2,3,4$ (again we have $n_4=0$). Again then $I_{P_i}^{n_i+m}$ sits $n_i+m$ units of length $|I_{\vec{P}}|$ away from $I_{P_i}$.

If we now fix $n_1, n_2, n_3 \in \mathbb{Z}$ and $\gamma \in [0,1]$ then it is sufficient to study estimates for the following discrete variant of \eqref{form}

\begin{equation*}
\sum\limits_{\vec{P} \in \vec{\mathbf{P}}}\frac{1}{|I_{\vec{P}}|}\langle f_1 , \phi_{P_1^{n_1}, 1} \rangle \langle f_2 , \phi_{P_2^{n_2}, 2} \rangle \langle f_3 , \phi_{P_3^{n_3}, 3} \rangle \langle f_4 , \phi_{P_4, 4} \rangle
\end{equation*}

\noindent Write

\begin{equation}\label{discr_form}
\Lambda_{\vec{\mathbf{P}}}(f_1,f_2,f_3,f_4) := \sum\limits_{\vec{P} \in \vec{\mathbf{P}}}\frac{1}{|I_{\vec{P}}|}\langle f_1 , \phi_{P_1^{n_1}, 1} \rangle \langle f_2 , \phi_{P_2^{n_2}, 2} \rangle \langle f_3 , \phi_{P_3^{n_3}, 3} \rangle \langle f_4 , \phi_{P_4, 4} \rangle
\end{equation}

\noindent and define $T_{\vec{\mathbf{P}}}(f_1,f_2,f_3)$ with

\begin{equation*}
\langle T_{\vec{\mathbf{P}}}(f_1,f_2,f_3), f_4 \rangle = \Lambda_{\vec{\mathbf{P}}}(f_1,f_2,f_3,f_4)
\end{equation*}

To compare our quadtiles with the tiles one faces in the bilinear Hilbert transform then notice that if $\vec{P} = (P_1, P_2, P_3, P_4)$ then $P_1$ is like a paraproduct tile, $P_2$ and $P_3$ might at a first glance seem just as in the bilinear Hilbert transform and $P_4$ is essentially as in the bilinear Hilbert transform, just potentially translated a bit in frequency by $P_1$. Note that the constant in the definition of $\lesssim'$ is $5$ as opposed to $3$ in \cite{MTT2}. We choose a bigger constant to make up for this extra possible translation of $P_4$. In the next section we will see in which cases we are essentially as in the bilinear Hilbert transform case, and in which cases we have to be more careful.

\section{Rank $(1,0)$}\label{rank}

Recall a standard definition of rank \cite{MTT2}.

\begin{definition}
A collection $\vec{\mathbf{P}}$ of quadtiles is said to have rank $1$ if one has the following properties for all $\vec{P}, \vec{P}' \in \vec{\mathbf{P}}$:
\begin{itemize}
\item If $\vec{P} \neq \vec{P}'$, then $P_j \neq P_j'$ for all j=1,2,3,4.
\item If $P_{j}' \leq P_j$ for some $j=1,2,3,4$, then $P_{i}' \lesssim P_i$ for all $1 \leq i \leq 4$.
\item If we further assume that $10^9|I_{\vec{P}'}| < |I_{\vec{P}}|$, then we have $P_{i}' \lesssim' P_i$ for all $i\neq j$.
\end{itemize}
\end{definition}

\noindent This definition does not work for our collection of quadtiles because the paraproduct tile $P_1$ does not uniquely determine the other three tiles.

We only need a frequency or time interval from one of our tiles to determine $P_1$, while we need a whole tile $P_j$, $j=2,3\text{ or }4$, to determine the other three. Motivated by this fact and what ingredients are really important in a rank definition \cite{MTT4} we give the following definition.

\begin{definition}
Let $\lbrace i_1, i_2, i_3, i_4 \rbrace$ be some rearrangement of $\lbrace 1,2,3,4 \rbrace$. A collection $\vec{\mathbf{P}}$ of quadtiles is said to have rank $(1,0)$ with respect to $\lbrace \lbrace i_1, i_2, i_3 \rbrace, \lbrace i_4 \rbrace \rbrace$ if one has the following properties for all $\vec{P}, \vec{P}' \in \vec{\mathbf{P}}$:
\begin{itemize}
\item If $\vec{P} \neq \vec{P}'$, then $P_{i_j} \neq P_{i_j}'$ for all j=1,2,3 and if $I_{\vec{P}} = I_{\vec{P}'}$ then $P_{i_4} = P_{i_4}'$.
\item If $P_{i_{j}}' \leq P_{i_j}$ for some $j=1,2,3$, then $P_{i_{k}}' \lesssim P_{i_k}$ for all $1 \leq k \leq 4$.
\item If we further assume that $10^9|I_{\vec{P}'}| < |I_{\vec{P}}|$, then there exist at least two indices
$$\tau_1(i_j), \tau_2(i_j) \in\lbrace 1,2,3,4\rbrace\setminus\lbrace i_j\rbrace,\quad \tau_1(i_j)\neq \tau_2(i_j)$$
such that we have $P_{\tau_1(i_j)}' \lesssim' P_{\tau_1(i_j)}$ and $P_{\tau_2(i_j)}' \lesssim' P_{\tau_2(i_j)}$. We call those indices good indices with respect to $i_j$ and note that there might be up to three of them. Here we understand $P_{i_4}' \lesssim' P_{i_4}$ to mean $\omega_{P_{i_4}'} \cap \omega_{P_{i_4}} = \emptyset$.
\end{itemize}
\end{definition}

Note that the orderings $\leq$ and $\lesssim'$ do not make sense for our paraproduct tiles because we have the relation $\leq$ between any two such tiles and thus $\lesssim'$ never happens. These orderings work well on the bilinear Hilbert transform type tiles where flexibility is helpful. We have to be more exact with the paraproduct tiles and thus understand the relation $\leq$ to mean that the paraproduct tiles intersect in frequency while $\lesssim'$ means that they don't intersect. 

It is not hard to see that our collection of quadtiles is rank $(1,0)$ with respect to $\lbrace \lbrace 2,3,4 \rbrace, \lbrace 1 \rbrace \rbrace$ where a collection corresponds to exactly one of the three cases we have. The first and second conditions are clearly fulfilled since knowing one of the bilinear Hilbert transform tiles gives us complete information about all the other tiles and since the paraproduct tile is completely determined by the time interval. Modulo a finite refinement of our collection we can also see that the last condition is fulfilled.

Assume we are in the case \eqref{case1} and that we have $10^9|I_{\vec{P}'}| < |I_{\vec{P}}|$ and $P_{2}' \leq P_{2}$. We cannot guarantee that $P_{3}' \leq P_{3}$ since $P_{2}$ and $P_{3}$ are essentially the same tile and similarly for $P_{2}'$ and $P_{3}'$. However $10^9|I_{\vec{P}'}| < |I_{\vec{P}}|$ guarantees that $\omega_{P_{1}'} \cap \omega_{P_{1}} = \emptyset$ which along with the previous observation also guarantees that $P_{4}' \lesssim' P_{4}$. The other possibilities in this case go somewhat similarly. This particular example shows how critical the paraproduct tile is in our analysis.

In the case \eqref{case2} then $P_{1}$ has minimal effect so we are essentially in the bilinear Hilbert case so all the conditions above are fulfilled.

Assume we are in the case \eqref{case3} and that we have $10^9|I_{\vec{P}'}| < |I_{\vec{P}}|$ and $P_{4}' \leq P_{4}$. We claim that $P_{2}' \lesssim' P_{2}$ and $P_{3}' \lesssim' P_{3}$ so let us assume for contradiction that $P_{2}' \leq P_{2}$. The distance between the centers of the frequency supports of $P_{1}$ and $P_{1}'$ is roughly $|\omega_{P_{1}'}| - |\omega_{P_{1}}| < |\omega_{P_{1}'}|$ which means, since $P_{2}' \leq P_{2}$ and $P_{4}' \leq P_{4}$, that the distance between the centers of the frequency supports of $P_{3}$ and $P_{3}'$ is at most $|\omega_{P_{1}'}|$ which gives $P_{3}' \leq P_{3}$. This must be a contradiction and thus we have $P_{2}' \lesssim' P_{2}$ and $P_{3}' \lesssim' P_{3}$. The other possibilities in this case go somewhat similarly.

\section{Fourier Coefficient}

Recall from \eqref{fourier_coeff} that the Fourier coefficient $C_{n_1,n_2,n_3}^{\vec{Q}}$ is given by

\begin{multline*}
C_{n_1, n_2, n_3}^{\vec{Q}} = \frac{1}{|\vec{Q}|^4} \int_{\mathbb{R}^3}\frac{ \left[\int_{0}^{1}1_{\mathbb{R}_{+}}(\alpha\xi_1 \! + \! \beta\xi_2 \! + \! \xi_3)d\alpha\right]}{a(\xi_1, \xi_2, \xi_3)}\widehat{\tilde{\phi}_{Q_1, 1}}(\xi_1)\widehat{\tilde{\phi}_{Q_2, 2}}(\xi_2)\widehat{\tilde{\phi}_{Q_3, 3}}(\xi_3) \\
e^{-2\pi i \frac{n_1}{|\vec{Q}|}\xi_1}e^{-2\pi i \frac{n_2}{|\vec{Q}|}\xi_2}e^{-2\pi i \frac{n_3}{|\vec{Q}|}\xi_3} d\xi_1 d\xi_2 d\xi_3
\end{multline*}

Change variables and obtain

\begin{multline*}
C_{n_1, n_2, n_3}^{\vec{Q}} = \int_{\mathbb{R}^3}\left[\int_{0}^{1}1_{\mathbb{R}_{+}}(\alpha\xi_1 \! + \! \beta\xi_2 \! + \! \xi_3)d\alpha\right]\frac{\widehat{\phi_1}(\xi_1)\widehat{\phi_2}(\xi_2)\widehat{\phi_3}(\xi_3)}{\tilde{a}(\xi_1, \xi_2, \xi_3)}e^{-2\pi i n_1 \xi_1} \\
e^{-2\pi i n_2 \xi_2}e^{-2\pi i n_3 \xi_3} d\xi_1 d\xi_2 d\xi_3
\end{multline*}

\noindent where $\widehat{\phi_i}(\xi_i) = \widehat{\tilde{\phi}_{Q_1, 1}}(|\vec{Q}|\xi_1)$ is a bump that is of scale $1$ and $\tilde{a}(\xi_1, \xi_2, \xi_3) = a(|\vec{Q}|\xi_1, |\vec{Q}|\xi_2,|\vec{Q}|\xi_3)$ is also of scale $1$ on the support of $\widehat{\phi_1}(\xi_1)\widehat{\phi_2}(\xi_2)\widehat{\phi_3}(\xi_3)$. To see why the last statement is true we have to recall

\begin{equation*}
a(\xi_1, \xi_2, \xi_3) = \sum\limits_{\vec{\tilde{Q}} \in \vec{\mathbf{Q}}}\widehat{\phi_{\tilde{Q}_1, 1}}(\xi_1)\widehat{\phi_{\tilde{Q}_2, 2}}(\xi_2)\widehat{\phi_{\tilde{Q}_3, 3}}(\xi_3)
\end{equation*}

\noindent and split into cases based on \eqref{case1}, \eqref{case2} and \eqref{case3}. First note that for a term in

\begin{equation*}
\sum\limits_{\vec{\tilde{Q}} \in \vec{\mathbf{Q}}}\widehat{\phi_{\tilde{Q}_1, 1}}(\xi_1)\widehat{\phi_{\tilde{Q}_2, 2}}(\xi_2)\widehat{\phi_{\tilde{Q}_3, 3}}(\xi_3)
\end{equation*}

\noindent to contribute to the sum on the support of $\vec{Q}$ we must have $\tilde{Q}_i \cap Q_i \neq \emptyset$ for $i=1,2,3$.

Start with the cases \eqref{case1} and \eqref{case3}. For $\tilde{Q}_1 \cap Q_1 \neq \emptyset$ we must have $|\vec{\tilde{Q}}| \sim |\vec{Q}|$ because else $\widehat{\phi_{\tilde{Q}_1, 1}}$ and $\widehat{\phi_{Q_1, 1}}$ have disjoint supports.

The last case is \eqref{case2}. Assume we have $\vec{\tilde{Q}}$ and $\vec{Q}$ such that $\tilde{Q}_i \cap Q_i \neq \emptyset$ for $i=1,2,3$. Let's now for symmetry assume we have $|\vec{\tilde{Q}}| \ll |\vec{Q}|$. We are in the case where $Q_2$ and $Q_3$ are several units of length $|\vec{Q}|$ away from one another and $\tilde{Q}_2$ and $\tilde{Q}_3$ are several units of length $|\vec{\tilde{Q}}|$ away from one another. However if $\tilde{Q}_2 \cap Q_2 \neq \emptyset$ then we can't have $\tilde{Q}_3 \cap Q_3 \neq \emptyset$ which is a contradiction. Thus we must have $|\vec{\tilde{Q}}| \sim |\vec{Q}|$.

We now want to integrate by parts to obtain decay in $n_1$, $n_2$, $n_3$. We do not need to worry about derivatives hitting $\frac{\widehat{\phi_1}(\xi_1)\widehat{\phi_2}(\xi_2)\widehat{\phi_3}(\xi_3)}{\tilde{a}(\xi_1, \xi_2, \xi_3)}$ which is smooth and of scale $1$.

In the case \eqref{case2} we do not catch the planes where our symbol is continuous but not differentiable. In that case we can thus integrate by parts as often as we want and obtain as much decay in $n_1$, $n_2$ and $n_3$ as we want.

In the other cases, \eqref{case1} and \eqref{case2}, we might catch the planes where our symbol is merely continuous but in both cases we know that $Q_1$ is away from the origin. Thus we can write $C_{n_1, n_2, n_3}^{\vec{Q}}$ as

\begin{multline}\label{four_coeff_scale1}
\int_{\mathbb{R}^3}\left[\frac{1}{\xi_1}\int_{0}^{\xi_1}1_{\mathbb{R}_{+}}(\alpha \! + \! \beta\xi_2 \! + \! \xi_3)d\alpha\right]\frac{\widehat{\phi_1}(\xi_1)\widehat{\phi_2}(\xi_2)\widehat{\phi_3}(\xi_3)}{\tilde{a}(\xi_1, \xi_2, \xi_3)}e^{-2\pi i n_1 \xi_1} \\
e^{-2\pi i n_2 \xi_2}e^{-2\pi i n_3 \xi_3} d\xi_1 d\xi_2 d\xi_3 \\
= \int_{\mathbb{R}^3}\left[\int_{0}^{\xi_1}1_{\mathbb{R}_{+}}(\alpha \! + \! \beta\xi_2 \! + \! \xi_3)d\alpha\right]\frac{\widehat{\tilde{\phi}_1}(\xi_1)\widehat{\phi_2}(\xi_2)\widehat{\phi_3}(\xi_3)}{\tilde{a}(\xi_1, \xi_2, \xi_3)}e^{-2\pi i n_1 \xi_1} \\
e^{-2\pi i n_2 \xi_2}e^{-2\pi i n_3 \xi_3} d\xi_1 d\xi_2 d\xi_3
\end{multline}

\noindent where $\widehat{\tilde{\phi}_1}(\xi_1) = \frac{1}{\xi_1}\widehat{\phi_1}(\xi_1)$ is well defined and still smooth because $\xi_1$ is always away from zero. As in Muscalu's treatment of the symbol for the Calder\'{o}n commutator \cite{M1}, which has a non-standard symbol, we get the following lemmas.

\begin{lemma}\label{diff_symbol}
One has the following identities
\begin{itemize}
\item[a)] $\partial_{\xi_3}^{2}\left(\int_{0}^{\xi_1}1_{\mathbb{R}_{+}}(\alpha \! + \! \beta\xi_2 \! + \! \xi_3)d\alpha \right) = \delta_{0}(\xi_1 + \beta\xi_2 + \xi_3) - \delta_0(\beta\xi_2 + \xi_3)$
\item[b)] $\partial_{\xi_2}\partial_{\xi_3}\left(\int_{0}^{\xi_1}1_{\mathbb{R}_{+}}(\alpha \! + \! \beta\xi_2 \! + \! \xi_3)d\alpha \right) = \beta( \delta_{0}(\xi_1 + \beta\xi_2 + \xi_3) - \delta_0(\beta\xi_2 + \xi_3))$
\item[c)] $\partial_{\xi_1}\partial_{\xi_3}\left(\int_{0}^{\xi_1}1_{\mathbb{R}_{+}}(\alpha \! + \! \beta\xi_2 \! + \! \xi_3)d\alpha \right) = \delta_{0}(\xi_1 + \beta\xi_2 + \xi_3)$
\item[d)] $\partial_{\xi_2}^{2}\left(\int_{0}^{\xi_1}1_{\mathbb{R}_{+}}(\alpha \! + \! \beta\xi_2 \! + \! \xi_3)d\alpha \right) = \beta^2(\delta_{0}(\xi_1 + \beta\xi_2 + \xi_3) - \delta_0(\beta\xi_2 + \xi_3))$
\item[e)] $\partial_{\xi_1}\partial_{\xi_2}\left(\int_{0}^{\xi_1}1_{\mathbb{R}_{+}}(\alpha \! + \! \beta\xi_2 \! + \! \xi_3)d\alpha \right) = \beta \delta_{0}(\xi_1 + \beta\xi_2 + \xi_3)$
\item[f)] $\partial_{\xi_1}^{2}\left(\int_{0}^{\xi_1}1_{\mathbb{R}_{+}}(\alpha \! + \! \beta\xi_2 \! + \! \xi_3)d\alpha \right) = \delta_{0}(\xi_1 + \beta\xi_2 + \xi_3)$
\end{itemize}
\end{lemma}

\begin{proof}
This is straightforward. Let us verify a) for instance. One has

\begin{align*}
\partial_{\xi_3}^{2}\left(\int_{0}^{\xi_1}1_{\mathbb{R}_{+}}(\alpha \! + \! \beta\xi_2 \! + \! \xi_3)d\alpha \right) &= \partial_{\xi_3}\left(\int_{0}^{\xi_1}\delta_0(\alpha \! + \! \beta\xi_2 \! + \! \xi_3)d\alpha \right) \\
&=  \partial_{\xi_3}\left(\int_{\beta\xi_2 + \xi_3}^{\xi_1 + \beta\xi_2 + \xi_3}\delta_0(\alpha) d\alpha \right) \\
&= \delta_{0}(\xi_1 + \beta\xi_2 + \xi_3) - \delta_0(\beta\xi_2 + \xi_3)
\end{align*}
\end{proof}

\begin{lemma}\label{decay}
\begin{multline*}
|C_{n_1, n_2, n_3}^{\vec{Q}}| \lesssim c_{\beta}^1 \frac{1}{\langle n_3 \rangle^2}\cdot \frac{1}{\langle n_1 - n_3 \rangle^{M_2}} \cdot \frac{1}{\langle n_2 - \beta n_3 \rangle^{M_3}} + c_{\beta}^2 \frac{1}{\langle n_3 \rangle^2}\cdot \frac{1}{\langle n_1 \rangle^{M_2}} \cdot \frac{1}{\langle n_2 - \beta n_3 \rangle^{M_3}}\\
+ c_{\beta}^{3}\frac{1}{\langle n_3 \rangle^{M_1}}\cdot \frac{1}{\langle n_1 - \frac{n_2}{\beta} \rangle^{M_2}} \cdot \frac{1}{\langle n_3 - \frac{n_2}{\beta} \rangle^{M_3}} + c_{\beta}^{4}\frac{1}{\langle n_3 \rangle^{M_1}}\cdot \frac{1}{\langle n_1 \rangle^{M_2}} \cdot \frac{1}{\langle n_3 - \frac{n_2}{\beta} \rangle^{M_3}} \\
+ c_{\beta}^{5}\frac{1}{\langle n_3 \rangle^{M_1}}\cdot \frac{1}{\langle n_2 \rangle^{M_2}} \cdot \frac{1}{\langle n_3 - n_1 \rangle^{M_3}}\cdot \frac{1}{\langle n_2 - \beta n_1 \rangle^{M_4}} + c_{\beta}^{6}\frac{1}{\langle n_3 \rangle^{M_1}}\cdot \frac{1}{\langle n_2 \rangle^{M_2}} \cdot \frac{1}{\langle n_1 \rangle^{M_3}}
\end{multline*}

\noindent where $\langle n \rangle := 2 + |n|$ and $M_1$, $M_2$, $M_3$, $M_4$ are fixed large integers and $c_{\beta}^{1}, \ldots , c_{\beta}^{6}$ are constants that only depend on $\beta$.
\end{lemma}

\begin{proof}
As mentioned before then this clearly holds in the case \eqref{case2} since then the symbol is smooth and we can integrate by parts as often as we want in the Fourier coefficient. In the other two cases \eqref{case1} and \eqref{case3} we must use lemma \ref{diff_symbol}. The idea is to integrate by parts in \eqref{four_coeff_scale1} in the $\xi_3$ variable as often as we can. Since both $\int_{0}^{\xi_1}1_{\mathbb{R}_{+}}(\alpha \! + \! \beta\xi_2 \! + \! \xi_3)d\alpha$ and $\frac{\widehat{\tilde{\phi}_1}(\xi_1)\widehat{\phi_2}(\xi_2)\widehat{\phi_3}(\xi_3)}{\tilde{a}(\xi_1, \xi_2, \xi_3)}$ depend on $\xi_3$ then derivatives can hit either of the terms. If the derivative hits the term $\int_{0}^{\xi_1}1_{\mathbb{R}_{+}}(\alpha \! + \! \beta\xi_2 \! + \! \xi_3)d\alpha$ twice then because of lemma \ref{diff_symbol} the $\xi_3$ variable disappears and \eqref{four_coeff_scale1} collapses to

\begin{multline*}
\int_{\mathbb{R}^2}\frac{\widehat{\tilde{\phi}_1}(\xi_1)\widehat{\phi_2}(\xi_2)\widehat{\phi_3}(-\xi_1 - \beta\xi_2)}{\tilde{a}(\xi_1, \xi_2, -\xi_1 - \beta\xi_2)}e^{-2\pi i (n_1 - n_3) \xi_1}
e^{-2\pi i (n_2 - \beta n_3) \xi_2} d\xi_1 d\xi_2 \\
- \int_{\mathbb{R}^2}\frac{\widehat{\tilde{\phi}_1}(\xi_1)\widehat{\phi_2}(\xi_2)\widehat{\phi_3}(-\beta\xi_2)}{\tilde{a}(\xi_1, \xi_2, -\xi_1 - \beta\xi_2)}e^{-2\pi i n_1 \xi_1}
e^{-2\pi i (n_2 - \beta n_3) \xi_2} d\xi_1 d\xi_2
\end{multline*}

The integrands in both those terms are smooth and can be integrated by parts as many times as we wish and all the derivatives are compactly supported on scale $1$. This explains the appearance of the first two terms in the estimate for $C_{n_1, n_2, n_3}^{\vec{Q}}$.

If however the $\xi_3$ derivative didn't hit the term $\int_{0}^{\xi_1}1_{\mathbb{R}_{+}}(\alpha \! + \! \beta\xi_2 \! + \! \xi_3)d\alpha$ two times, even after running the procedure many times, this means that we already gained a factor of the type $\frac{1}{\langle n_3 \rangle^{M_1}}$, at which point we stop integrating by parts in $\xi_3$ and start integrating by parts in $\xi_2$. If $\xi_2$ derivatives hit $\int_{0}^{\xi_1}1_{\mathbb{R}_{+}}(\alpha \! + \! \beta\xi_2 \! + \! \xi_3)d\alpha$ we face two possible cases, we either end up with  $\partial_{\xi_2}\partial_{\xi_3}\left(\int_{0}^{\xi_1}1_{\mathbb{R}_{+}}(\alpha \! + \! \beta\xi_2 \! + \! \xi_3)d\alpha \right)$ or $\partial_{\xi_2}^{2}\left(\int_{0}^{\xi_1}1_{\mathbb{R}_{+}}(\alpha \! + \! \beta\xi_2 \! + \! \xi_3)d\alpha \right)$. Using lemma \ref{diff_symbol} then the integral collapses as in the first case, that is $\xi_2$ becomes $-\frac{\xi_3 + \xi_1}{\beta}$ or $-\frac{\xi_3}{\beta}$. After that we are, as before, integrating by parts a smooth function, obtaining an upper bound that explains the appearance of the third and fourth terms in the estimate for $C_{n_1, n_2, n_3}^{\vec{Q}}$.

If however $\int_{0}^{\xi_1}1_{\mathbb{R}_{+}}(\alpha \! + \! \beta\xi_2 \! + \! \xi_3)d\alpha$ has not been hit two times by some combination of $\xi_3$ and $\xi_2$ derivatives after running the procedure many times, this means that we have already gained a factor of the type $\frac{1}{\langle n_3 \rangle^{M_1}} \cdot \frac{1}{\langle n_2 \rangle^{M_2}}$ at which point we stop integrating by parts in $\xi_2$ and start integrating by parts in $\xi_1$. If $\xi_1$ derivatives hit $\int_{0}^{\xi_1}1_{\mathbb{R}_{+}}(\alpha \! + \! \beta\xi_2 \! + \! \xi_3)d\alpha$ we face three possible cases, we end up with $\partial_{\xi_1}\partial_{\xi_3}\left(\int_{0}^{\xi_1}1_{\mathbb{R}_{+}}(\alpha \! + \! \beta\xi_2 \! + \! \xi_3)d\alpha \right)$, $\partial_{\xi_1}\partial_{\xi_2}\left(\int_{0}^{\xi_1}1_{\mathbb{R}_{+}}(\alpha \! + \! \beta\xi_2 \! + \! \xi_3)d\alpha \right)$ or $\partial_{\xi_1}^{2}\left(\int_{0}^{\xi_1}1_{\mathbb{R}_{+}}(\alpha \! + \! \beta\xi_2 \! + \! \xi_3)d\alpha \right)$. Using lemma \ref{diff_symbol} the integral collapses as before, that is $\xi_1$ becomes $-\beta\xi_2 - \xi_3$. After that we are, as before, integrating by parts a smooth function, obtaining an upper bound that explains the appearance of the fifth term in the estimate for $C_{n_1, n_2, n_3}^{\vec{Q}}$.

Last but not least, if no combination of $\xi_1$, $\xi_2$ or $\xi_3$ derivatives hits $\int_{0}^{\xi_1}1_{\mathbb{R}_{+}}(\alpha \! + \! \beta\xi_2 \! + \! \xi_3)d\alpha$ twice then this means that the derivatives keep hitting the smooth function in which case we obtain an upper bound that explains the appearance of the last term in the estimate for $C_{n_1, n_2, n_3}^{\vec{Q}}$.

\end{proof}

\section{Discrete Operator}

Let now $\vec{\mathbf{P}}$ be a finite collection of multitiles which is sparse and has rank $(1,0)$. Consider also wave packets $(\phi_{P_{j}^{n_j}, j})_{\vec{P} \in \vec{\mathbf{P}}}$ for $j=1,2,3,4$ adapted to the tiles $P_{j}^{n_j}$ respectively as before where $n_1, n_2 \text{ and } n_3$ are fixed and $n_4=0$. Assume also that they are all $L^2$-normalized. The following theorem will be proven in detail in section 9.

\begin{theorem}\label{rest_weak_type}
Let $\gamma_1$ and $\gamma_3$ be positive numbers, smaller than $1$ but very close to $1$, $\gamma_2$ be a positive number smaller than $\frac{1}{2}$ but very close to $\frac{1}{2}$. Let also $E_1$, $E_2$, $E_3$, $E_4 \subseteq \mathbb{R}$ be measurable sets of finite measure. Then there exists $E_{4}' \subseteq E_4$ with $|E_{4}'| \sim |E_4|$ such that for every $|f_1| \lesssim 1_{E_1}$, $|f_2| \lesssim 1_{E_2}$, $|f_3| \lesssim 1_{E_3}$ one has

\begin{equation}
\Bigl\lvert \int_{\mathbb{R}}T_{\vec{\mathbf{P}}}(f_1,f_2,f_3)(x) 1_{E_{4}'}(x) dx \Bigr\rvert \lesssim \left(\prod\limits_{j=1}^{3}|\log_2(\langle n_j \rangle )|^4 \right)|E_1|^{\gamma_1}|E_2|^{\gamma_2}|E_3|^{\gamma_3}|E_4|^{\gamma_4}
\end{equation}

\noindent where $\gamma_4$ is defined by $\gamma_1 + \gamma_2 + \gamma_3 + \gamma_4 = 1$. Moreover the implicit constant is independent of the cardinality of $\vec{\mathbf{P}}$.
\end{theorem}

Using the interpolation theory by Muscalu, Tao and Thiele \cite{MTT4}, the symmetries of $T_{\vec{\mathbf{P}}}$ and standard duality arguments then one can deduce the following theorem.

\begin{theorem}\label{6.2}
If $\vec{\mathbf{P}}$ is as before then $T_{\vec{\mathbf{P}}}$ maps boundedly

\begin{equation}
T_{\vec{\mathbf{P}}} : L^{p_1}(\mathbb{R})\times L^{p_2}(\mathbb{R}) \times L^{p_3}(\mathbb{R}) \mapsto L^{p_4}(\mathbb{R})
\end{equation}

\noindent for any $1 <p_1, p_2, p_3 \leq \infty$ and $\frac{2}{5} <p_4<\infty$ such that $\frac{1}{p_1} + \frac{1}{p_2} + \frac{1}{p_3} = \frac{1}{p_4}$ and $\frac{2}{3}< \frac{p_2 p_3}{p_2 + p_3} \leq \infty$. Furthermore, the constant of boundedness depends on $n_1, n_2, n_3 \text{ and } n_4$ in a way that can be bounded by $\prod\limits_{j=1}^{3}|\log_2(\langle n_j \rangle )|^4$.
\end{theorem}

\noindent Note that this is a stronger result than in theorem \ref{original_thm}.

To prove Theorem \ref{original_thm} then let $p_1$, $p_2$, $p_3$ and $p_4$ be as in the theorem and recall that in section 3 we commented that it is enough to show the theorem for $\tilde{T}_{\beta}$. If $p_4 \geq 1$ then standard arguments extend the theorem to $\tilde{T}_{\beta}$. If however $p_4 < 1$ let $f_i \in L^{p_i}(\mathbb{R})$, $i=1,2,3$ and note

\begin{align*}
\| \tilde{T}_{\beta}(f_1,f_2,f_3)\|_{p_4} &= \| \tilde{T}_{\beta}(f_1,f_2,f_3)^{p_4} \|_{1}^{1/p_4} \\
&\lesssim \| \left( \int_{0}^{1}\sum\limits_{n_1, n_2, n_3}C(n_1,n_2,n_3) T_{\vec{\mathbf{P}},\eta}(f_1, f_2, f_3)d\eta \right)^{p_4} \|_{1}^{1/p_4} \\
&\lesssim \| \int_{0}^{1}\sum\limits_{n_1, n_2, n_3}C(n_1,n_2,n_3)^{p_4} T_{\vec{\mathbf{P}},\eta}^{p_4}(f_1, f_2, f_3)d\eta \|_{1}^{1/p_4}
\end{align*}

This last step is only well defined if $p_4 > \frac{1}{2}$ because $C(n_1, n_2, n_3)$ includes terms that contain $\frac{1}{\langle n_3 \rangle^2}$ by lemma \ref{decay} and we need $\frac{1}{\langle n_3 \rangle^{2p_4}}$ to be summable. In that case then theorem \ref{6.2} and lemma \ref{decay}, along with standard results on the convergence of series of the type $\sum\limits_{n}\frac{|\log_2(\langle n \rangle )|^4}{\langle n \rangle^p}$ where $p>1$, can be used to conclude that theorem \ref{original_thm} holds true for $\tilde{T}_{\beta}$ and thus for $T_{\beta}$.

Note that the reason why $p_4>\frac{1}{2}$ might seem a bit naive. One could hope to improve the result by treating all the $T_{\vec{\mathbf{P}}}$ simultaneously by picking in section \ref{proofDiscrOp} a common exceptional set. Such a strategy leads to a loss of $\langle n \rangle^{(1+\epsilon)}$, $\epsilon >0$, in the size estimates in section \ref{tileNorms}. Thus, running through the standard argument, one would eventually have to control a sum of the following type
$$\sum\limits_{n_1,n_2,n_3}C(n_1,n_2,n_3)\langle n_1 \rangle^{(1+\epsilon)b_1}\langle n_2 \rangle^{(1+\epsilon)b_2}\langle n_3 \rangle^{(1+\epsilon)b_3} .$$
If we only consider the first term in the estimate of the Fourier coefficient one faces the following sum
$$\sum\limits_{n_1,n_2,n_3}\frac{\langle n_3 \rangle^{(1+\epsilon)b_3}}{\langle n_3 \rangle^2}\cdot \frac{\langle n_1 \rangle^{(1+\epsilon)b_1}}{\langle n_1 - n_3 \rangle^{M_2}} \cdot \frac{\langle n_2 \rangle^{(1+\epsilon)b_2}}{\langle n_2 - \beta n_3 \rangle^{M_3}} .$$
Changing variables through
$$ k= n_1 - n_3 \text{ and } l= n_2 - \beta n_3 $$
the sum becomes
$$\sum\limits_{k,l,n_3}\frac{\langle n_3 \rangle^{(1+\epsilon)b_3}}{\langle n_3 \rangle^2}\cdot \frac{\langle k + n_3 \rangle^{(1+\epsilon)b_1}}{\langle k \rangle^{M_2}} \cdot \frac{\langle l+\beta n_3 \rangle^{(1+\epsilon)b_2}}{\langle l \rangle^{M_3}}.$$
Hence, one would like the expression
$$ \frac{\langle n_3 \rangle^{(1+\epsilon)(b_1 + b_2 + b_3)}}{\langle n_3 \rangle^2} .$$
to be summable, which places stringent requirements on $b_1$, $b_2$ and $b_3$. In fact, if one goes thoroughly through the standard argument it is not hard to see that the condition $p_4 > \frac{1}{2}$ cannot be improved. It is thus an interesting open question whether this condition can be improved, which clearly either requires some novel ideas or some more delicate estimates.

\section{Trees}

The standard approach to prove the desired estimates for the form $\Lambda_{\vec{\mathbf{P}}}$ is to organize the collection of quadtiles $\vec{\mathbf{P}}$ into trees. We may assume, and will do so for the rest of the article, that $\vec{\mathbf{P}}$ is sparse and of rank $(1,0)$. We will now recall basic definitions and comments for trees from \cite{MTT2}. The only change is that we will not consider $1$ trees at all. We will essentially ignore the first position when setting up the trees. Also note that we set up the trees based on untranslated tiles.

\begin{definition}
For any $2\leq j \leq 4$ and a quadtile $\vec{P}_T \in \vec{\mathbf{P}}$, define a $j$-tree with top $\vec{P}_T$ to be a collection of quadtiles $T\subseteq\vec{\mathbf{P}}$ such that

\begin{equation}
P_j \leq P_{T,j} \text{ for all } \vec{P}\in T,
\end{equation}

\noindent where $P_{T,j}$ is the $j$ component of $\vec{P}_T$. We write $I_T$ and $\omega_{T,j}$ for $I_{\vec{P}_T}$ and $\omega_{P_{T,j}}$ respectively. We say that $T$ is a tree if it is a $j$-tree for some $2 \leq j \leq 4$.
\end{definition}

Note that $T$ does not necessarily have to contain its top $\vec{P}_{T}$.

\begin{definition}
Let $2 \leq i \leq 4$. Two trees $T$, $T'$ are said to be strongly $i$-disjoint if
\begin{itemize}
\item $P_i \neq P_{i}'$ for all $\vec{P}\in T$, $\vec{P'}\in T'$.
\item Whenever $\vec{P}\in T$, $\vec{P'} \in T'$ are such that $2\omega_{P_i}\cap 2\omega_{P_{i}'} \neq \emptyset$, then one has $I_{\vec{P'}} \cap I_T = \emptyset$, and similarly with $T$ and $T'$ reversed.
\end{itemize}
\end{definition}

Note that if $T$ and $T'$ are strongly $i$-disjoint, then $I_P \times 2\omega_{P_i}\cap I_{P'} \times 2\omega_{P_{i}'} = \emptyset$ for all $\vec{P}\in T$, $\vec{P'} \in T'$.

Given that $\vec{\mathbf{P}}$ is sparse, it is easy to see that if $T$ is an $i$-tree, then for all $\vec{P}, \vec{P'} \in T$ and $j \neq i$, $2\leq j \leq 4$, we have

\begin{equation*}
\omega_{P_j} = \omega_{P_{j}'}
\end{equation*}

\noindent or

\begin{equation*}
2\omega_{P_j}\cap 2\omega_{P_{j}'} = \emptyset
\end{equation*}

We pick trees for tiles $\vec{P}$ as in the bilinear Hilbert transform case but remember that our wave packets are in general adapted to tiles $P_{i}^{n_i}$, $i=1,2,3$, that are translated in time by $n_i$ units of length $|I_{\vec{P}}|$. Thus the effective trees we face are translated and are furthermore not evenly translated.

Due to the dyadic structure of the trees and the dyadic structure of the translation applied to the tiles in the trees then one can see that we can do better than saying that a translated tree, derived from a tree $T$, is supported on $\bigcup\limits_{j=0}^{n_i}I_{T}^{j}$. As Muscalu observes \cite{M1} (and can be seen from the argument in section \ref{ShiftedMax}) then in fact the translated tree is supported on $\bigcup\limits_{j\in Fr(n_i)}I_{T}^{j}$ where $Fr(n_i)$ is a set of indices that contains for example $0$, $1$ and $n_i$. We also know the following fact about the cardinality of $Fr(n_i)$

\begin{equation*}
|Fr(n_i)| \lesssim \log_2(\langle n_i \rangle) .
\end{equation*}

\noindent We call $\bigcup\limits_{j\in Fr(n_i)}I_{T}^{j}$ "$I_T$ and friends".

\section{Tile Norms}\label{tileNorms}

Let's recall the standard tile norms from the paper by Muscalu, Tao and Thiele \cite{MTT2}.

\begin{definition}
Let $\vec{\mathbf{P}}$ be a finite collection of quadtiles, $j=1,2,3,4$ and let $(a_{P_j})_{\vec{P}\in\vec{\mathbf{P}}}$ be a sequence of complex numbers. We define the size of this sequence by

\begin{equation*}
\text{size}_j((a_{P_j})_{\vec{P}\in\vec{\mathbf{P}}}) := \sup\limits_{T\subset \vec{\mathbf{P}}}(\frac{1}{|I_T|}\sum\limits_{\vec{P}\in T}|a_{P_j}|^2)^{1/2}
\end{equation*}

\noindent where $T$ ranges over all trees in $\vec{\mathbf{P}}$ which are either one quadtile trees or $i$-trees for some $2\leq i \leq 4$ such that $j$ is a good index with respect to $i$, as in the definition of rank $(1,0)$.

We also define the energy of a sequence by

\begin{equation*}
\text{energy}_j((a_{P_j})_{\vec{P}\in\vec{\mathbf{P}}}) := \sup\limits_{n\in\mathbb{Z}}\sup\limits_{\mathbb{T}} 2^n (\sum\limits_{T\in\mathbf{T}}|I_T|)^{1/2}
\end{equation*}

\noindent where $\mathbf{T}$ ranges over all collections of strongly $j$-disjoint trees, $2 \leq j \leq 4$, in $\vec{\mathbf{P}}$ such that

\begin{equation*}
(\sum\limits_{\vec{P}\in T}|a_{P_j}|^2)^{1/2} \geq 2^n |I_T|^{1/2}
\end{equation*}

\noindent for all $T\in\mathbf{T}$ and

\begin{equation*}
(\sum\limits_{\vec{P}\in T'}|a_{P_j}|^2)^{1/2} \leq 2^{n+1} |I_{T'}|^{1/2}
\end{equation*}

\noindent for all sub-trees $T' \subset T\in\mathbf{T}$.

\end{definition}

\noindent We will use those definitions for $a_{P_j} = \langle f_j, \phi_{P_j^{n_j}, j} \rangle$. Note that the restriction to $i$-trees for some $2\leq i \leq 4$ such that $j$ is a good index with respect to $i$, as in the definition of rank $(1,0)$, means that whenever such trees exist then we can attempt to use square function estimates on our collection of $P_j$ tiles that come with those trees. In other words, the $P_j$ tiles stack up similarly as in the bilinear Hilbert transform case.

Recall the John-Nirenberg inequality \cite{MTT2}.

\begin{lemma}\label{john_nirenberg}
Let $\vec{\mathbf{P}}$ be a finite collection of quadtiles, $j=1,2,3,4$ and let $(a_{P_j})_{\vec{P}\in\vec{\mathbf{P}}}$ be a sequence of complex numbers. Then

\begin{equation*}
\text{size}_j((a_{P_j})_{\vec{P}\in\vec{\mathbf{P}}}) \sim \sup\limits_{T\subset \vec{\mathbf{P}}}\frac{1}{|I_T|}\|(\sum\limits_{\vec{P}\in T}|a_{P_j}|^2\frac{1_{I_{\vec{P}}}}{|I_{\vec{P}}|})^{1/2}\|_{L^{1,\infty}(I_T)}
\end{equation*}

\noindent where $T$ ranges over all trees in $\vec{\mathbf{P}}$ which are either one quadtile trees or $i$-trees for some $2\leq i \leq 4$ such that $j$ is a good index with respect to $i$, as in the definition of rank $(1,0)$.
\end{lemma}

\noindent The proof carries exactly over due to our choice of possible trees in the definition of size.

\section{Proof of Discrete Operator Theorem}\label{proofDiscrOp}

\begin{proposition}\label{size_energy_est}
Let $\vec{\mathbf{P}}$ be a finite collection of quadtiles. Then

\begin{multline*}
|\Lambda_{\vec{\mathbf{P}}}(f_1,f_2,f_3,f_4)| \lesssim \text{ size}((\langle f_1, \phi_{P_1^{n_1}, 1} \rangle)_{\vec{P}\in\vec{\mathbf{P}}})\prod\limits_{j=2}^{4}(\text{size}((\langle f_j, \phi_{P_j^{n_j}, j}\rangle )_{\vec{P}\in\vec{\mathbf{P}}}))^{\theta_j} \\
(\text{energy}((\langle f_j, \phi_{P_j^{n_j}, j} \rangle )_{\vec{P}\in\vec{\mathbf{P}}}))^{1-\theta_j}
\end{multline*}

\noindent for any $0 \leq \theta_2, \theta_3, \theta_4 < 1$ with $\theta_2 + \theta_3 + \theta_4 = 1$, with the implicit constant depending on the $\theta_i$.
\end{proposition}

\noindent This proposition will be proven in section 14.

\begin{lemma}\label{size_est}
Let $\vec{\mathbf{P}}$ be a finite collection of quadtiles, $j \in \lbrace 1,2,3,4 \rbrace$ and $E$ be a set of finite measure. Then for every $|f| \leq 1_{E}$ one has

\begin{equation*}
\text{size}((\langle f, \phi_{P_j^{n_j}, j} \rangle)_{\vec{P}\in\vec{\mathbf{P}}}) \lesssim \log_2(\langle n_j \rangle) \sup\limits_{\vec{P}\in\vec{\mathbf{P}}}\frac{1}{|I_{\vec{P}}|}\int_{E}\tilde{\chi}_{I_{P_j^{n_j}}}^{M}
\end{equation*}

\noindent for all $M>0$, with the implicit constant depending on $M$.
\end{lemma}

\noindent Lemma \ref{size_est} will be proven in section 12.

Define the shifted dyadic maximal operator $M^{n}$ as follows \cite{M1}

\begin{equation*}
M^{n}f(x) := \sup\limits_{x\in I}\frac{1}{|I|}\int_{\mathbb{R}}|f(y)|\tilde{\chi}_{I^n}(y)dy
\end{equation*}

\noindent where the supremum is taken only over dyadic intervals.

\begin{lemma}\label{shifted_max}
For any $n\in\mathbb{Z}$ the shifted maximal function $M^n$ maps boundedly $L^p(\mathbb{R})$ into $L^p(\mathbb{R})$ with a bound of the type $O(log_2(\langle n \rangle))$. It also maps boundedly $L^{\infty}(\mathbb{R})$ into $L^{\infty}(\mathbb{R})$ and $L^1(\mathbb{R})$ into $L^{1,\infty}(\mathbb{R})$ with a bound of the type $O(log_2(\langle n \rangle))$.
\end{lemma}

\noindent Lemma \ref{shifted_max} will be proven in section 11.

\begin{lemma}\label{energy_est}
Let $\vec{\mathbf{P}}$ be a finite collection of quadtiles, $j \in \lbrace 2,3,4 \rbrace$ and $f\in L^2(\mathbb{R})$. Then

\begin{equation*}
\text{energy}((\langle f, \phi_{P_j^{n_j}, j} \rangle)_{\vec{P}\in\vec{\mathbf{P}}}) \lesssim (\log_2(\langle n_j \rangle))^2 \|f\|^2
\end{equation*}
\end{lemma}

\noindent Lemma \ref{energy_est} will be proven in section 13.

We can now prove theorem \ref{rest_weak_type}. 

\begin{proof}
Fix $E_1$, $E_2$, $E_3$, $E_4$, $\gamma_1$, $\gamma_2$ and $\gamma_3$ as in the hypothesis of theorem \ref{rest_weak_type}. The goal is to find $E_{4}' \subseteq E_4$ with $|E_{4}'| \sim |E_4|$ such that for every $|f_1| \lesssim 1_{E_1}$, $|f_2| \lesssim 1_{E_2}$, $|f_3| \lesssim 1_{E_3}$ one has

\begin{equation*}
\Bigl\lvert \Lambda_{\vec{\mathbf{P}}}(f_1,f_2,f_3, 1_{E_{4}'}) \Bigr\rvert \lesssim \left(\prod\limits_{j=1}^{3}|\log_2(\langle n_j \rangle )|^4 \right)|E_1|^{\gamma_1}|E_2|^{\gamma_2}|E_3|^{\gamma_3}|E_4|^{\gamma_4}
\end{equation*}

\noindent where we recall that $\gamma_4$ is defined by $\gamma_1 + \gamma_2 + \gamma_3 + \gamma_4 = 1$.

Using the dilation symmetry of $T_{\beta}$, which translates naturally to $\Lambda_{\vec{\mathbf{P}}}$, one can clearly assume wlog that $|E_4| = 1$. Define then the set $\Omega$ by

\begin{equation*}
\Omega := \bigcup\limits_{j=1}^{3}\left( \lbrace x: M^{n_j}\left(\frac{1_{E_j}}{|E_j|}\right)(x) > C\log_2(\langle n_j \rangle) \rbrace \right)
\end{equation*}

\noindent and observe that $|\Omega| \ll 1$ if $C$ is a large enough constant. Then set $E_{4}' := E_4 \setminus \Omega$ and notice that $|E_{4}'| \sim 1$ as desired.

Then for any $d\geq 1$ define the collection $\vec{\mathbf{P}}_d$ by

\begin{equation*}
\vec{\mathbf{P}}_d := \lbrace \vec{P}\in\vec{\mathbf{P}} : 2^{d-1} \leq \frac{dist(I_{\vec{P}},\Omega^c)}{|I_{\vec{P}}|} \leq 2^{d} \rbrace
\end{equation*}

\noindent and let $P_0$ be the collection of quadtiles which intersect $\Omega^c$. Clearly $\bigcup\limits_{d\geq 0}\vec{\mathbf{P}}_d = \vec{\mathbf{P}}$.

We can write

\begin{equation}\label{inner_form}
 \Lambda_{\vec{\mathbf{P}}}(f_1,f_2,f_3, 1_{E_{4}'}) = \sum\limits_{d=0}^{\infty}\int_{\mathbb{R}}T_{\vec{\mathbf{P}}_d}(f_1,f_2,f_3)(x)1_{E_{4}'}(x)dx
\end{equation}

Fix $d\geq 0$ and consider the inner quad linear form of \eqref{inner_form}. It can be estimated by proposition \ref{size_energy_est}. Using lemma \ref{size_est} and lemma \ref{shifted_max} we obtain

\begin{align*}
\text{size}((\langle f, \phi_{P_j^{n_j}, j} \rangle)_{\vec{P}\in\vec{\mathbf{P}}_d}) &\lesssim \log_2(\langle n_j \rangle) \sup\limits_{\vec{P}\in\vec{\mathbf{P}}_d}\frac{1}{|I_{\vec{P}}|}\int_{E}\tilde{\chi}_{I_{P_j^{n_j}}}^{M} \\
&\lesssim (\log_2(\langle n_j \rangle))^2\min(1,2^d|E_j|) \\
&\lesssim (\log_2(\langle n_j \rangle))^2 2^d |E_j|^{a_j}
\end{align*}

\noindent for any $0 < a_j < 1$, $j=1,2,3$.

Using lemma \ref{energy_est} we also obtain for $j=2,3$

\begin{equation*}
\text{energy}((\langle f, \phi_{P_j^{n_j}, j} \rangle)_{\vec{P}\in\vec{\mathbf{P}}}) \lesssim (\log_2(\langle n_j \rangle))^2 |E_j|^{1/2}.
\end{equation*}

Using lemmas \ref{size_est}, \ref{energy_est} and \ref{shifted_max} for the fourth position, using $n_4=0$, we note that since $|E_4|=1$ we obtain

\begin{equation*}
\text{size}((\langle f, \phi_{P_j^{n_j}, j} \rangle)_{\vec{P}\in\vec{\mathbf{P}}_d}) \lesssim 2^{-Md}
\end{equation*}

\noindent and

\begin{equation*}
\text{energy}((\langle f, \phi_{P_j^{n_j}, j} \rangle)_{\vec{P}\in\vec{\mathbf{P}}}) \lesssim 1.
\end{equation*}

Putting all this together then proposition \ref{size_energy_est} allows us to bound the corresponding quad linear form in \eqref{inner_form} for a fixed $d\geq 0$ by

\begin{multline*}
2^{-\# d}|E_1|^{a_1}(|E_2|^{a_2})^{\theta_2}(|E_2|^{1/2})^{1 - \theta_2}(|E_3|^{a_3})^{\theta_2}(|E_3|^{1/2})^{1 - \theta_3}\cdot 1 \\
= 2^{-\# d}|E_1|^{a_1}|E_2|^{a_2 \theta_2 + \frac{1}{2}(1-\theta_2)}|E_3|^{a_3\theta_3 + \frac{1}{2}(1-\theta_3)}
\end{multline*}

\noindent where $\#$ is a strictly positive integer. Then we can make $a_1$ arbitrarily close to $1$, $a_2 \theta_2 + \frac{1}{2}(1-\theta_2)$ arbitrarily close to $\frac{1}{2}$ by choosing $\theta_2$ close to $0$ and $a_3\theta_3 + \frac{1}{2}(1-\theta_3)$ arbitrarily close to $1$ by choosing $\theta_3$ close to $1$ and $a_3$ also close to $1$.
\end{proof}

\section{Estimates for the shifted dyadic maximal function}\label{ShiftedMax}

We will now recall the proof of lemma \ref{shifted_max} from \cite{M1}. We note, as Muscalu does in \cite{M1}, that the proof of this lemma was already known and can be found in \cite{S1} Chapter II.

\begin{proof}
Observe that it is sufficient to prove the estimates for the "sharp" shifted dyadic maximal function $\tilde{M}^{n}$ defined by

\begin{equation*}
\tilde{M}^{n}f(x) := \sup\limits_{x\in I}\frac{1}{|I|}\int_{I^n}|f(y)|dy
\end{equation*}

where the supremum is only taken over dyadic intervals.

To observe this, fix $x$ and $I$ so that $x\in I$. We can write

\begin{equation*}
\frac{1}{|I^n|}\int_{I^n}|f(y)|dy \lesssim \sum\limits_{\# \in \mathbb{Z}}\left[\frac{1}{|I^{n+\#}|}\int_{I^{n+\#}}|f(y)|dy \right]\frac{1}{\langle \# \rangle^{100}}.
\end{equation*}

\noindent Assuming the theorem holds for $\tilde{M}^{n}$ and using the above, one has

\begin{align*}
\|M^{n}f \|_{p} &\lesssim \sum\limits_{\# \in \mathbb{Z}}\frac{1}{\langle \# \rangle^{100}}\|\tilde{M}^{n+\#}f \|_{p}\\
&\lesssim \sum\limits_{\# \in \mathbb{Z}}\frac{1}{\langle \# \rangle^{100}}\log_2(\langle n+\# \rangle)\|f \|_{p} \\
&\lesssim \sum\limits_{\# \in \mathbb{Z}}\frac{1}{\langle \# \rangle^{100}}\log_2(\langle n \rangle \langle \# \rangle)\|f \|_{p}\\
& \lesssim \log_2(\langle n \rangle)\|f \|_{p}
\end{align*}

\noindent as desired. We then turn to proving the theorem for $\tilde{M}^n$.

Let $\lambda > 0$. We claim that the following inequality is true

\begin{equation}\label{11claim}
|\lbrace x: \tilde{M}^{n}f(x) > \lambda \rbrace| \lesssim \log_2(\langle n \rangle)|\lbrace x: Mf(x) > \lambda \rbrace|
\end{equation}

\noindent where $M$ is the classical Hardy-Littlewood maximal operator. Assuming \eqref{11claim} the theorem for $\tilde{M}^{n}$ follows from the Hardy-Littlewood theorem in the case $L^{1}(\mathbb{R}) \mapsto L^{1,\infty}(\mathbb{R})$. The case $L^{\infty}(\mathbb{R}) \mapsto L^{\infty}(\mathbb{R})$ is trivial. All the other estimates we obtain then by interpolating between those two cases.

To prove \eqref{11claim} denote by $\mathcal{I}_{n}^{\lambda}$ the collection of all dyadic and maximal, with respect to inclusion, intervals $I^n$, for which

\begin{equation*}
\frac{1}{|I^n|}\int_{I^n}|f(y)|dy > \lambda.
\end{equation*}

Observe they are all disjoint and in addition one has

\begin{equation*}
\bigcup\limits_{I^n \in \mathcal{I}_{n}^{\lambda}}I^n = \lbrace x: Mf(x) > \lambda \rbrace .
\end{equation*}

\noindent For every such selected, maximal, dyadic interval $I^n$, then it has at most $\log_2(\langle n \rangle)$ friends as in the tree case. More precisely then there are at most $\log_2(\langle n \rangle)$ disjoint dyadic intervals $I_{1}^{n}, \ldots ,I_{N}^{n}$ of the same length as $|I^n|$, so that the translate with $-n$ corresponding units of any subinterval of $I^n$ becomes a subinterval of one of these intervals. Now we claim

\begin{equation*}
\lbrace x: \tilde{M}^{n}f(x) > \lambda \rbrace \subseteq \bigcup\limits_{I^n \in \mathcal{I}_{n}^{\lambda}}(I_{1}^{n} \cup \ldots \cup I_{1}^{N}).
\end{equation*}

\noindent To prove this, pick $x^{*}$ such that $M^n f(x^{*}) > \lambda$. This implies that there exists a dyadic interval $J$ containing $x^{*}$ such that $\frac{1}{|J^n|}\int_{J^n}|f(y)|dy > \lambda$. Due to the previous construction, one can certainly find one selected maximal interval of the type $I^n$ such that $J^n \subseteq I^n$. This however means that $J$ itself will be a subset of one of $I_{1}^{n}, \ldots ,I_{N}^{n}$ which proves the claim.

One can now easily see that this claim and the disjointness of the maximal intervals $I^n$ along with the fact that $N \leq \log_2(\langle n \rangle)$ imply \eqref{11claim}.

\end{proof}

\section{Size estimates}

We will now prove lemma \ref{size_est}.

\begin{proof}
Fix $j\in\lbrace 1,2,3,4 \rbrace$, $n_j$, $E$ and $|f|\lesssim 1_{E}$ as in the lemma. Since $\vec{\mathbf{P}}$ is a finite set of tiles there exists a tree $\tilde{T}$  such that the supremum in the size is attained. If the tree is just one quadtile then the proof is trivial. Let's thus assume that $\tilde{T}$ is an $i$-tree for some $2\leq i \leq 4$ such that $j$ is a good index with respect to $i$, as in the definition of rank $(1,0)$.

\begin{multline}\label{11.1}
\text{size}((\langle f, \phi_{P_j^{n_j}, j} \rangle)_{\vec{P}\in\vec{\mathbf{P}}}) = (\frac{1}{|I_{\tilde{T}}|}\sum\limits_{\vec{P}\in\tilde{T}}|\langle f, \phi_{P_j^{n_j}, j} \rangle|^2)^{1/2} \\
\leq \sum\limits_{i\in Fr(n_j)}(\frac{1}{|I_{\tilde{T}}|}\sum\limits_{\substack{\vec{P}\in\tilde{T}\\ I_{\vec{P}}\subseteq I_{\tilde{T}}^{i}}}|\langle f, \phi_{P_j^{n_j}, j} \rangle|^2)^{1/2}
\end{multline}

\noindent Now for each $i\in Fr(n_j)$ take $\vec{P}\in\tilde{T}$ such that $I_{\vec{P}}\subseteq I_{\tilde{T}}^{i}$ and pick from that collection of tiles trees that are maximal with regards to inclusion and such that they contain their top. Call that collection $\vec{T}_i$ for each $i\in Fr(n_j)$. Then we can bound \eqref{11.1} with

\begin{equation*}
 \sum\limits_{i\in Fr(n_j)}\sum\limits_{T\in \vec{T}_i} (\frac{1}{|I_{\tilde{T}}|}\sum\limits_{\vec{P}\in T}|\langle f, \phi_{P_j, j} \rangle|^2)^{1/2}
\end{equation*}

\noindent Note that the trees in $\vec{T}_i$ are disjoint and in particular

\begin{equation*}
\sum\limits_{T\in\vec{T}_i}|I_T| \leq |I_{\tilde{T}}|.
\end{equation*}

\noindent Thus for a fixed $i\in Fr(n_j)$ we have

\begin{align*}
\sum\limits_{T\in \vec{T}_i} (\frac{1}{|I_{\tilde{T}}|}\sum\limits_{\vec{P}\in T}|\langle f, \phi_{P_j, j} \rangle|^2)^{1/2} &\leq \left(\sup\limits_{T \in \vec{T}_i}\left(\frac{1}{|I_{T}|}\sum\limits_{\vec{P}\in T}|\langle f, \phi_{P_j, j} \rangle|^2 \right)^{1/2} \right)\frac{1}{|I_{\tilde{T}}|}\sum\limits_{T\in\vec{T}_i}|I_T| \\
&\leq \sup\limits_{T \in \vec{T}_i}\left(\frac{1}{|I_{T}|}\sum\limits_{\vec{P}\in T}|\langle f, \phi_{P_j, j} \rangle|^2 \right)^{1/2}
\end{align*}

Since $\vec{\mathbf{P}}$ is a finite set of tiles then for each friend there exists a tree $T$ which is an $i$-tree for some $i\neq j$, $2 \leq i \leq 4$, such that 

\begin{equation*}
\sup\limits_{T \in \vec{T}_i}\left(\frac{1}{|I_{T}|}\sum\limits_{\vec{P}\in T}|\langle f, \phi_{P_j, j} \rangle|^2 \right)^{1/2} \sim \frac{1}{|I_T|}\left\|\left(\sum\limits_{\vec{P}\in T}|\langle f, \phi_{P_j, j} \rangle|^2\frac{1_{I_{\vec{P}}}}{|I_{\vec{P}}|}\right)^{1/2}\right\|_{1,\infty}
\end{equation*}

\noindent Here we have also used the John-Nirenberg inequality in lemma \ref{john_nirenberg}. Clearly it is enough to prove that

\begin{equation*}
\left\|\left(\sum\limits_{\vec{P}\in T}|\langle f, \phi_{P_j, j} \rangle|^2\frac{1_{I_{\vec{P}}}}{|I_{\vec{P}}|}\right)^{1/2}\right\|_{1,\infty} \lesssim \int_{\mathbb{R}}1_{E_j}\tilde{\chi}_{I_T}^{M}
\end{equation*}

\noindent and use the fact that $|Fr(n_j)| \leq \log_2(\langle n_j \rangle)$.

Decompose the real line as a union of intervals

\begin{equation*}
\mathbb{R} = \bigcup\limits_{n\in\mathbb{Z}}I_{T}^n
\end{equation*}

\noindent where $|I_{T}^{n}| = |I_T|$ for every $n\in\mathbb{Z}$, $I_{T}^{0} = I_T$ and all $I_{T}^{n}$ are disjoin except for the endpoints. We think of $I_{T}^{n}$ as being n units of length $|I_T|$ to the right of $I_T$ if $n>0$ and to the left if $n<0$. Then split $f$ as

\begin{equation*}
f = f\cdot 1_{5I_T} + f \cdot 1_{(5I_T)^c}.
\end{equation*}

\noindent Since the expression $\left(\sum\limits_{\vec{P}\in T}|\langle f, \phi_{P_j, j} \rangle|^2\frac{1_{I_{\vec{P}}}}{|I_{\vec{P}}|}\right)^{1/2}$ is a square function, it is bounded from $L^1$ into $L^{1,\infty}$ and as a consequence

\begin{equation*}
\left\|\left(\sum\limits_{\vec{P}\in T}|\langle f\cdot 1_{5I_T}, \phi_{P_j, j} \rangle|^2\frac{1_{I_{\vec{P}}}}{|I_{\vec{P}}|}\right)^{1/2}\right\|_{1,\infty} \lesssim \|f\cdot 1_{5I_T}\|_1
\end{equation*}

\noindent which can be majorized by the expression in the right-hand side of the lemma.

We are left with estimating

\begin{equation*}
\left\|\left(\sum\limits_{\vec{P}\in T}|\langle f\cdot 1_{(5I_T)^c}, \phi_{P_j, j} \rangle|^2\frac{1_{I_{\vec{P}}}}{|I_{\vec{P}}|}\right)^{1/2}\right\|_{1,\infty}
\end{equation*}

\noindent which is clearly smaller than

\begin{equation*}
\sum\limits_{|n|\geq 3}\sum\limits_{\vec{P}\in T}\frac{\langle |f\cdot 1_{I_{T}^{n}}|,|\phi_{P_j, j}| \rangle}{|I_P|^{1/2}}|I_P| \lesssim \sum\limits_{|n|\geq 3}\sum\limits_{\vec{P}\in T}\langle |f|\cdot 1_{I_{T}^{n}},|\tilde{\chi}_{I_{\vec{P}}}^M| \rangle
\end{equation*}

\noindent for any big number $M>0$. In order to complete the proof it is enough to prove that

\begin{equation*}
\sum\limits_{\vec{P}\in T}\langle |f|\cdot 1_{I_{T}^{n}},|\tilde{\chi}_{I_{\vec{P}}}^M| \rangle \lesssim \frac{1}{\langle n \rangle^M}\int_{\mathbb{R}}1_{E_j}1_{I_{T}^{n}}
\end{equation*}

\noindent but this is an easy consequence of the fact that the sum on the left-hand side runs over $P$ for which $I_{P} \subseteq I_T$. This ends the proof of lemma \ref{size_est}.

\end{proof}

\section{Energy estimates}

We will now prove lemma \ref{energy_est}.

\begin{proof}
Fix $j\in\lbrace 1,2,3,4 \rbrace$ and $f\in L^{2}(\mathbb{R})$. Let also $n$ and $\mathbf{T}$ be as in definition of energy such that the supremum in the definition is attained. We want to show that

\begin{equation}\label{12.1}
2^n \left(\sum\limits_{T\in\mathbf{T}}|I_T| \right)^{1/2} \lesssim \|f\|_2
\end{equation}

\noindent If we square the left-hand side of \eqref{12.1} and use the properties of the trees in $\mathbf{T}$ we can write

\begin{multline*}
\left( 2^n \left(\sum\limits_{T\in\mathbf{T}}|I_T| \right)^{1/2} \right)^2 = 2^{2n}\sum\limits_{T\in\mathbf{T}}|I_T| \\
\lesssim 2^{2n} 2^{-2n}\sum\limits_{T\in\mathbf{T}}\left( \sum\limits_{\vec{P}\in T}|\langle f, \phi_{P_{j}^{n_j}, j} \rangle|^2 \right) = \sum\limits_{T\in\mathbf{T}}\left( \sum\limits_{\vec{P}\in T}|\langle f, \phi_{P_{j}^{n_j}, j} \rangle|^2 \right)
\end{multline*}

\noindent and this expression is supposed to be smaller than $\|f\|_2^2$. We can also write

\begin{multline*}
\sum\limits_{T\in\mathbf{T}}\sum\limits_{\vec{P}\in T}|\langle f, \phi_{P_{j}^{n_j}, j} \rangle|^2 = |\langle \sum\limits_{T\in\mathbf{T}}\sum\limits_{\vec{P}\in T}\langle f, \phi_{P_{j}^{n_j}, j} \rangle \phi_{P_{j}^{n_j}, j} ,f \rangle| \\
\lesssim \| f \|_2 \| \sum\limits_{T\in\mathbf{T}}\sum\limits_{\vec{P}\in T}\langle f, \phi_{P_{j}^{n_j}, j} \rangle \phi_{P_{j}^{n_j}, j} \|_2
\end{multline*}

\noindent so it is enough to prove that

\begin{equation}\label{11.2}
\| \sum\limits_{T\in\mathbf{T}}\sum\limits_{\vec{P}\in T}\langle f, \phi_{P_{j}^{n_j}, j} \rangle \phi_{P_{j}^{n_j}, j} \|_2 \lesssim \left( \sum\limits_{T\in\mathbf{T}}\sum\limits_{\vec{P}\in T}|\langle f, \phi_{P_{j}^{n_j}, j} \rangle|^2 \right)^{1/2}
\end{equation}

\noindent The square of the left-hand side of \eqref{11.2} becomes smaller than

\begin{equation}\label{11.3}
\sum\limits_{T, T' \in\mathbf{T}}\sum\limits_{\substack{\vec{P}\in T \\ \vec{Q} \in T'}}|\langle f, \phi_{P_{j}^{n_j}, j} \rangle||\langle f, \phi_{Q_{j}^{n_j}, j} \rangle||\langle \phi_{P_{j}^{n_j}, j}, \phi_{Q_{j}^{n_j}, j} \rangle| := \text{ I } + \text{ II }
\end{equation}

\noindent where I contains the part where $T\neq T'$ while II contains the $T=T'$ part.

We first estimate I. Observe that if $\vec{P}\in T$ and $\vec{Q}\in T'$ then, in order for $\langle \phi_{P_{j}^{n_j}, j}, \phi_{Q_{j}^{n_j}, j} \rangle$ to be non-zero, we must have $\omega_{P_j} \cap \omega_{Q_j} \neq \emptyset$ and so we either have $\omega_{P_j} \subseteq \omega_{Q_j}$ or $\omega_{Q_j} \subseteq \omega_{P_j}$. Because of the symmetry we can assume that we always have $\omega_{P_j} \subseteq \omega_{Q_j}$. Then, since $T$ and $T'$ are strictly disjoint, this means that $I_{\vec{Q}} \cap I_{T} = \emptyset$ for any such a $\vec{Q}$.

Fix now $T$, $T'$, $\vec{P} \in T$ and $\vec{Q} \in T'$ so that $\omega_{P_j} \subseteq \omega_{Q_j}$. Using the properties of the trees $T\in\mathbf{T}$, we can write

\begin{equation*}
\frac{1}{|I_{\vec{P}}|^{1/2}}|\langle f, \phi_{P_{j}^{n_j}, j} \rangle| \lesssim 2^n \lesssim \frac{1}{|I_T|^{1/2}}\left(\sum\limits_{\vec{\tilde{P}}}|\langle f, \phi_{\tilde{P}_{j}^{n_j}, j} \rangle|^2 \right)^{1/2}
\end{equation*}

\noindent from which we can deduce that

\begin{equation}\label{11.4}
|\langle f, \phi_{P_{j}^{n_j}, j} \rangle| \lesssim \frac{|I_{\vec{P}}|^{1/2}}{|I_T|^{1/2}}\left(\sum\limits_{\vec{\tilde{P}}}|\langle f, \phi_{\tilde{P}_{j}^{n_j}, j} \rangle|^2 \right)^{1/2}.
\end{equation}

\noindent Similarly we have

\begin{equation}\label{11.5}
|\langle f, \phi_{Q_{j}^{n_j}, j} \rangle| \lesssim \frac{|I_{\vec{Q}}|^{1/2}}{|I_T|^{1/2}}\left(\sum\limits_{\vec{\tilde{P}}}|\langle f, \phi_{\tilde{P}_{j}^{n_j}, j} \rangle|^2 \right)^{1/2}.
\end{equation}

\noindent Using \eqref{11.4} and \eqref{11.5} we can bound I in \eqref{11.3} with

\begin{multline}\label{11.6}
\sum\limits_{T, T' \in\mathbf{T}}\sum\limits_{\substack{\vec{P}\in T \\ \vec{Q} \in T' \\ \omega_{P_j}\subseteq \omega_{Q_j}}} \left[ \frac{|I_{\vec{P}}|^{1/2}}{|I_T|^{1/2}}\left(\sum\limits_{\vec{\tilde{P}}}|\langle f, \phi_{\tilde{P}_{j}^{n_j}, j} \rangle|^2 \right)^{1/2}\right] \left[\frac{|I_{\vec{Q}}|^{1/2}}{|I_T|^{1/2}}\left(\sum\limits_{\vec{\tilde{P}}}|\langle f, \phi_{\tilde{P}_{j}^{n_j}, j} \rangle|^2 \right)^{1/2} \right]\\
|\langle \phi_{P_{j}^{n_j}, j}, \phi_{Q_{j}^{n_j}, j} \rangle| \\
= \sum\limits_{T\in\mathbf{T}}\left( \sum\limits_{\vec{\tilde{P}}\in T}|\langle f, \phi_{\tilde{P}_{j}^{n_j}, j} \rangle|^2 \right) \sum\limits_{\vec{P}\in T}\sum\limits_{\substack{T' \in \mathbf{T}\\ T'\neq T}}\sum\limits_{\substack{\vec{Q}\in T' \\ \omega_{P_j}\subseteq\omega_{Q_j}}}\frac{1}{|I_T|}|I_{\vec{P}}|^{1/2}|I_{\vec{Q}}|^{1/2}|\langle \phi_{P_{j}^{n_j}, j}, \phi_{Q_{j}^{n_j}, j} \rangle| \\
\lesssim \sum\limits_{T\in\mathbf{T}}\left( \sum\limits_{\vec{\tilde{P}}\in T}|\langle f, \phi_{\tilde{P}_{j}^{n_j}, j} \rangle|^2 \right) \sum\limits_{\vec{P}\in T}\sum\limits_{\substack{T' \in \mathbf{T}\\ T'\neq T}}\sum\limits_{\substack{\vec{Q}\in T' \\ \omega_{P_j}\subseteq\omega_{Q_j}}}\frac{1}{|I_T|}|\langle \tilde{\chi}_{I_{P_{j}^{n_j}}}, \tilde{\chi}_{I_{Q_{j}^{n_j}}} \rangle|
\end{multline}

\noindent Fix $T$ and look at the corresponding inner sum in \eqref{11.6}.

\begin{equation}\label{11.7}
\sum\limits_{\vec{P}\in T}\sum\limits_{\substack{T' \in \mathbf{T}\\ T'\neq T}}\sum\limits_{\substack{\vec{Q}\in T' \\ \omega_{P_j}\subseteq\omega_{Q_j}}}\frac{1}{|I_T|}|\langle \tilde{\chi}_{I_{P_{j}^{n_j}}}, \tilde{\chi}_{I_{Q_{j}^{n_j}}} \rangle|
\end{equation}

\noindent It is clearly enough to show that this expression is $O((\log_2(\langle n_j \rangle))^2 |I_T|)$.

Fix $\vec{P}\in T$ and recall

\begin{equation*}
|\langle \tilde{\chi}_{I_{P_{j}^{n_j}}}, \tilde{\chi}_{I_{Q_{j}^{n_j}}} \rangle| \lesssim \left(1 + \frac{\text{dist}(I_{P_{j}^{n_j}},I_{Q_{j}^{n_j}})}{|I_{\vec{P}}|} \right)^{-M}|I_{\vec{Q}}|.
\end{equation*}

Set $\vec{\mathbf{Q}}_{\vec{P}} = \lbrace \vec{Q}\in T' : T'\in\mathbf{T}, T'\neq T, \omega_{P_{j}} \subseteq \omega_{Q_{j}} \rbrace$. Pick $\vec{\tilde{Q}}$ from $\vec{\mathbf{Q}}_{\vec{P}}$ such that $I_{\tilde{Q}_{j}^{n_j}}$ is maximal with respect to inclusion and place all $\vec{\tilde{\tilde{Q}}}\in \vec{\mathbf{Q}}_{\vec{P}}$ such that $I_{\tilde{\tilde{Q}}_{j}^{n_j}}\cap I_{\tilde{Q}_{j}^{n_j}} \neq \emptyset$ and $\vec{\tilde{\tilde{Q}}} \neq \vec{\tilde{Q}}$ into $S_{\vec{\tilde{Q}}}$. Then observe that

\begin{align*}
\sum\limits_{\vec{Q}\in S_{\vec{\tilde{Q}}}\cup\lbrace \tilde{Q}\rbrace} |\langle \tilde{\chi}_{I_{P_{j}^{n_j}}}, \tilde{\chi}_{I_{Q_{j}^{n_j}}} \rangle| &\lesssim \sum\limits_{\vec{Q}\in S_{\vec{\tilde{Q}}}\cup\lbrace \tilde{Q}\rbrace}\left(1 + \frac{\text{dist}(I_{P_{j}^{n_j}},I_{Q_{j}^{n_j}})}{|I_{\vec{P}}|} \right)^{-M}|I_{\vec{Q}}| \\
&\lesssim \left(1 + \frac{\text{dist}(I_{P_{j}^{n_j}},I_{\tilde{Q}_{j}^{n_j}})}{|I_{\vec{P}}|} \right)^{-M}\sum\limits_{\vec{Q}\in S_{\vec{\tilde{Q}}}\cup\lbrace \tilde{Q}\rbrace}|I_{\vec{Q}}|.
\end{align*}

\noindent Here we use the fact that $|I_{\vec{P}}|>|I_{\vec{Q}}|$ for all $\vec{Q}\in\vec{\mathbf{Q}}_{\vec{P}}$. Now note that the $I_{\vec{Q}}$ for all $\vec{Q}\in S_{\vec{\tilde{Q}}}$ are disjoint and they can only come from the friends of $I_{\vec{\tilde{Q}}}$ so

\begin{equation*}
\sum\limits_{\vec{Q}\in S_{\vec{\tilde{Q}}}\cup\lbrace \tilde{Q}\rbrace}|I_{\vec{Q}}| \lesssim \log_{2}(\langle n_j \rangle)|I_{\vec{\tilde{Q}}}|
\end{equation*}

\noindent Now place $\vec{\tilde{Q}}$ into $\vec{\mathbf{Q}}_{\vec{P}}^{*}$ and throw away $S_{\vec{\tilde{Q}}}\cup\vec{\tilde{Q}}$ from $\vec{\mathbf{Q}}_{\vec{P}}$ and iterate the selection process. Since $\mathbf{P}$ is finite then our selection process will take finitely many steps. We can bound \eqref{11.7} from above with

\begin{equation}\label{11.8}
\sum\limits_{\vec{P}\in T}\sum\limits_{\vec{Q}\in \vec{\mathbf{Q}}_{\vec{P}}^{*}}\log_{2}(\langle n_j \rangle)\left(1 + \frac{\text{dist}(I_{P_{j}^{n_j}},I_{Q_{j}^{n_j}})}{|I_{\vec{P}}|} \right)|I_{\vec{Q}}|
\end{equation}

\noindent where all the $I_{\vec{Q}}$ for $\vec{Q} \in \vec{\mathbf{Q}}_{\vec{P}}^{*}$ are disjoint.

Now split \eqref{11.8} in the following way

\begin{multline*}
\log_{2}(\langle n_j \rangle)\sum\limits_{\substack{\vec{P}\in T \\ 4n_j|I_{\vec{P}}| \geq |I_T|}}\sum\limits_{\vec{Q}\in \vec{\mathbf{Q}}_{\vec{P}}^{*}}\left(1 + \frac{\text{dist}(I_{P_{j}^{n_j}},I_{Q_{j}^{n_j}})}{|I_{\vec{P}}|} \right)|I_{\vec{Q}}| \\
+ \log_{2}(\langle n_j \rangle)\sum\limits_{\substack{\vec{P}\in T \\ 4n_j|I_{\vec{P}}| < |I_T|}}\sum\limits_{\vec{Q}\in \vec{\mathbf{Q}}_{\vec{P}}^{*}}\left(1 + \frac{\text{dist}(I_{P_{j}^{n_j}},I_{Q_{j}^{n_j}})}{|I_{\vec{P}}|} \right)|I_{\vec{Q}}|.
\end{multline*}

Pick all $\vec{P}\in T$ with $|I_{\vec{P}}|$ of the same length such that $4n_j|I_{\vec{P}}| \geq |I_T|$. Then for a fixed $\vec{P}$ we can estimate

\begin{equation*}
\sum\limits_{\vec{Q}\in \vec{\mathbf{Q}}_{\vec{P}}^{*}}\left(1 + \frac{\text{dist}(I_{P_{j}^{n_j}},I_{Q_{j}^{n_j}})}{|I_{\vec{P}}|} \right)|I_{\vec{Q}}| \lesssim |I_{\vec{P}}|
\end{equation*}

\noindent and since the $I_{\vec{P}}$ are all disjoint for $\vec{P}\in T$ of the same scale then when we add up $|I_{\vec{P}}|$ for all of them we get something less than $|I_T|$. Now note there are at most $O(\log_2(\langle n_j \rangle))$ scales of $\vec{P}$ such that $4n|I_{\vec{P}}| = 2^{\log_2(4n_j)}|I_{\vec{P}}| > |I_T|$ and thus

\begin{equation*}
\log_{2}(\langle n_j \rangle)\sum\limits_{\substack{\vec{P}\in T \\ 4n_j|I_{\vec{P}}| \geq |I_T|}}\sum\limits_{\vec{Q}\in \vec{\mathbf{Q}}_{\vec{P}}^{*}}\left(1 + \frac{\text{dist}(I_{P_{j}^{n_j}},I_{Q_{j}^{n_j}})}{|I_{\vec{P}}|} \right)|I_{\vec{Q}}| \lesssim (\log_{2}(\langle n_j \rangle))^2|I_T|.
\end{equation*}

Now look at $\vec{P}\in T$ with $4n_j|I_{\vec{P}}| < |I_T|$. Those $\vec{P}$, that are less than $3n_j$ units of length $|I_{\vec{P}}|$ away from the endpoints of $I_T$, might interact with $\vec{Q} \in \vec{\mathbf{Q}}_{\vec{P}}^{*}$ and for those we estimate

\begin{equation*}
\sum\limits_{\vec{Q}\in \vec{\mathbf{Q}}_{\vec{P}}^{*}}\left(1 + \frac{\text{dist}(I_{P_{j}^{n_j}},I_{Q_{j}^{n_j}})}{|I_{\vec{P}}|} \right)|I_{\vec{Q}}| \lesssim |I_{\vec{P}}|.
\end{equation*}

\noindent Note that for a given scale there are at most $6n_j$ of them. For those that are $l>3n_j$ units of length $|I_{\vec{P}}|$ away from the endpoints of $I_T$ then $I_{\vec{P}}\cap I_{\vec{Q}} = \emptyset$ for all $\vec{Q} \in \vec{\mathbf{Q}}_{\vec{P}}^{*}$. Thus we estimate

\begin{equation*}
\sum\limits_{\vec{Q}\in \vec{\mathbf{Q}}_{\vec{P}}^{*}}\left(1 + \frac{\text{dist}(I_{P_{j}^{n_j}},I_{Q_{j}^{n_j}})}{|I_{\vec{P}}|} \right)|I_{\vec{Q}}| \lesssim (1+(l-3n))^{-M} |I_{\vec{P}}|.
\end{equation*}

For a given such scale of $\vec{P}$, say $|I_{\vec{P}}|=2^k$, we get

\begin{align*}
&\log_{2}(\langle n_j \rangle)\sum\limits_{\substack{\vec{P}\in T \\ |I_{\vec{P}}|=2^k}}\sum\limits_{\vec{Q}\in \vec{\mathbf{Q}}_{\vec{P}}^{*}}\left(1 + \frac{\text{dist}(I_{P_{j}^{n_j}},I_{Q_{j}^{n_j}})}{|I_{\vec{P}}|} \right)|I_{\vec{Q}}| \\
&\hskip3em\lesssim (\log_{2}(\langle n_j \rangle))\left( 6n_j|I_{\vec{P}}| + |I_{\vec{P}}|\sum\limits_{l=3n+1}^{\infty}\frac{1}{(1+(l-3n_j))^{M}} \right) \\
&\hskip3em\lesssim \log_{2}(\langle n_j \rangle) (6n_j+1) |I_{\vec{P}}|
\end{align*}

\noindent Now if we sum up over all scales such that $|I_{\vec{P}}| < \frac{|I_T|}{4n_j}$ we get

\begin{align*}
\log_{2}(\langle n_j \rangle)\sum\limits_{\substack{\vec{P}\in T \\ 4n_j|I_{\vec{P}}| < |I_T|}}\sum\limits_{\vec{Q}\in \vec{\mathbf{Q}}_{\vec{P}}^{*}}\left(1 + \frac{\text{dist}(I_{P_{j}^{n_j}},I_{Q_{j}^{n_j}})}{|I_{\vec{P}}|} \right)|I_{\vec{Q}}| &\lesssim \log_{2}(\langle n_j \rangle) (6n_j+1)\frac{|I_T|}{4n_j} \\
&\lesssim \log_{2}(\langle n_j \rangle)|I_T|.
\end{align*}

We are now left with the diagonal term II from \eqref{11.3} where the sum runs over $T=T'$. If $\vec{P},\vec{Q}\in T$ and $\omega_{P_j}\cap\omega_{Q_j} \neq \emptyset$ then we must have $\omega_{P_j} = \omega_{Q_j}$. We can majorize II with

\begin{equation*}
\sum\limits_{T\in\mathbf{T}}\sum\limits_{\vec{P}\in T}|\langle f, \phi_{P_{j}^{n_j}, j} \rangle|^2 \frac{1}{|I_{\vec{P}}|}\left( \sum\limits_{\substack{\vec{Q}\in T\\ \omega_{P_j}= \omega_{Q_j}}} |\langle \tilde{\chi}_{I_{P_{j}^{n_j}}}, \tilde{\chi}_{I_{Q_{j}^{n_j}}} \rangle| \right)
\end{equation*}

\noindent and it is sufficient to show that

\begin{equation*}
\sum\limits_{\substack{\vec{Q}\in T\\ \omega_{P_j}= \omega_{Q_j}}} |\langle \tilde{\chi}_{I_{P_{j}^{n_j}}}, \tilde{\chi}_{I_{Q_{j}^{n_j}}} \rangle|
\end{equation*}

\noindent is $O(\log_{2}(\langle n_j \rangle)|I_{\vec{P}}|)$ but that follows immediately from the fact that all the $I_{\vec{Q}}$ for which $\omega_{P_j} = \omega_{Q_j}$ are disjoint.

This concludes the proof of lemma \ref{energy_est}.

\end{proof}

\section{Proof of proposition \ref{size_energy_est}}

We will now prove proposition \ref{size_energy_est}. Fix the collection $\vec{\mathbf{P}}$ of quadtiles and the functions $f_1$, $f_2$, $f_3$, $f_4$. As mentioned before then we assume that $\vec{\mathbf{P}}$ is sparse and of rank $(1,0)$ and assume it is with respect to $\lbrace \lbrace 2,3,4 \rbrace, \lbrace 1 \rbrace \rbrace$ without loss of generality.

Denote for simplicity

\begin{equation*}
S_j := \text{size}((\langle f, \phi_{P_j^{n_j}, j} \rangle)_{\vec{P}\in\vec{\mathbf{P}}})
\end{equation*}

\noindent for $j\in\lbrace 1,2,3,4 \rbrace$ and

\begin{equation*}
E_j := \text{energy}((\langle f, \phi_{P_j^{n_j}, j} \rangle)_{\vec{P}\in\vec{\mathbf{P}}})
\end{equation*}

\noindent for $j\in\lbrace 2,3,4 \rbrace$.

\begin{proposition}
Let $j\in \lbrace 2,3,4 \rbrace$ and $\vec{\mathbf{P}}' \subseteq \vec{\mathbf{P}}$, $n\in\mathbb{Z}$ so that

\begin{equation*}
\text{size}((\langle f, \phi_{P_j^{n_j}, j} \rangle)_{\vec{P}\in\vec{\mathbf{P}}'}) \leq 2^{-n}E_j .
\end{equation*}

\noindent Then one can decompose $\vec{\mathbf{P}}' = \vec{\mathbf{P}}'' \cup \vec{\mathbf{P}}'''$ such that

\begin{equation*}
\text{size}((\langle f, \phi_{P_j^{n_j}, j} \rangle)_{\vec{P}\in\vec{\mathbf{P}}''}) \leq 2^{-n-1}E_j
\end{equation*}

\noindent and $\vec{\mathbf{P}}'''$ can be written as a disjoint union of trees $T\in\mathbf{T}$ such that

\begin{equation*}
\sum\limits_{T\in\mathbf{T}}|I_T| \lesssim 2^{2n}
\end{equation*}

\end{proposition}

\begin{proof}
Our rank $(1,0)$ collection of quadtiles has all the relevant features in common with the collection of tritiles in the bilinear Hilbert transform so the proof from there works here.
\end{proof}

By iterating the previous result we obtain the following corollary.

\begin{corollary}
Let $\vec{\mathbf{P}}$ be a finite collection. Then one can split $\vec{\mathbf{P}}$ as

\begin{equation*}
\vec{\mathbf{P}} = \bigcup\limits_{n\in\mathbb{Z}}\vec{\mathbf{P}}_n
\end{equation*}

\noindent where for each $n\in\mathbb{Z}$ and $j=2,3,4$ we have

\begin{equation*}
\text{size}((\langle f, \phi_{P_j^{n_j}, j} \rangle)_{\vec{P}\in\vec{\mathbf{P}}_n}) \leq \min( 2^{-n}E_j, S_j ).
\end{equation*}

\noindent Also one can cover $\vec{\mathbf{P}}_n$ by a collection of trees $T\in\mathbf{T}_n$ for which

\begin{equation*}
\sum\limits_{T\in\mathbf{T}_n}|I_T| \lesssim 2^{2n}.
\end{equation*}

\end{corollary}

\begin{lemma}
Let $T$ be an $i$-tree, $i=2,3\text{ or }4$, in $\vec{\mathbf{P}}$ and $f_1$, $f_2$, $f_3$, $f_4$ fixed functions, then

\begin{equation*}
\sum\limits_{\vec{P}\in T}\frac{1}{|I_{\vec{P}}|}|\langle f_1 , \phi_{P_1^{n_1}, 1} \rangle| |\langle f_2 , \phi_{P_2^{n_2}, 2} \rangle| |\langle f_3 , \phi_{P_3^{n_3}, 3} \rangle| |\langle f_4 , \phi_{P_4^{n_4}, 4} \rangle| \leq |I_T| \prod\limits_{j=1}^{4}\text{size}((\langle f, \phi_{P_j^{n_j}, j} \rangle)_{\vec{P}\in T})
\end{equation*}

\end{lemma}

\begin{proof}
Say $T$ is a 2-tree and assume without loss of generality that $1$ and $4$ are good indices with respect to the index $2$. This is for example the case for our particular operator when we are in the case \eqref{case1} as discussed in section \ref{rank}. We can bound the left-hand side by

\begin{equation*}
\left(\sum\limits_{\vec{P}\in T} |\langle f_1 , \phi_{P_1^{n_1}, 1} \rangle|  \right)^{1/2}\left(\sup\limits_{\vec{P}\in T}\frac{|\langle f_2 , \phi_{P_2^{n_2}, 2} \rangle|}{|I_{\vec{P}}|^{1/2}}\right)\left(\sup\limits_{\vec{P}\in T}\frac{|\langle f_3 , \phi_{P_2^{n_3}, 3} \rangle|}{|I_{\vec{P}}|^{1/2}}\right)\left(\sum\limits_{\vec{P}\in T} |\langle f_4 , \phi_{P_4^{n_4}, 4} \rangle|  \right)^{1/2}
\end{equation*}

\noindent Since $1$ and $4$ are good indices with respect to $2$ we clearly have for $j=1,4$

\begin{equation*}
\left(\sum\limits_{\vec{P}\in T} |\langle f_j , \phi_{P_j^{n_j}, j} \rangle|  \right)^{1/2} \leq |I_T|^{1/2} \text{size}((\langle f, \phi_{P_j^{n_j}, j} \rangle)_{\vec{P}\in T}).
\end{equation*}

\noindent Since trees that consist of a single quadtile are also used in the definition of size then we clearly also have for $j=2,3$

\begin{equation*}
\sup\limits_{\vec{P}\in T}\frac{|\langle f_j , \phi_{P_j^{n_j}, j} \rangle|}{|I_{\vec{P}}|^{1/2}} \leq \text{size}((\langle f, \phi_{P_j^{n_j}, j} \rangle)_{\vec{P}\in T}).
\end{equation*}

\noindent In a similar manner one can verify the lemma for all other possible trees.

\end{proof}

We now have the tools to complete the proof of proposition \ref{size_energy_est}.

\begin{proof}
Using the corollary and lemma above then the proof runs as in the bilinear Hilbert transform case.
\end{proof}

\section{The Water Wave Problem}

In the $2$-d water wave problem, Wu showed that if one starts with small initial data then classical solutions exist for a long time \cite{W1}. In a natural way she came across operators of the following type

$$ f \mapsto p.v. \int\limits_{\mathbb{R}} F\left( \frac{A(x)-A(y)}{x-y} \right) \frac{\Pi_{i=1}^{n}(B_i(x)-B_i(y))}{(x-y)^{n+1}}f(y)\ dy $$

\noindent  and had to obtain $L^p$ estimates for them.  For such operators $L^p$ estimates are known if $A', B_i' \in L^{\infty}(\mathbb{R})$ for $i=1,\ldots,n$ and $f\in L^{2}(\mathbb{R})$. The novelty in Wu's paper was that she faced $B_1' \in L^{2}(\mathbb{R})$, which indicated that the operator should be viewed as a multilinear operator.

It is clear that operators similar to Wu's appear in PDEs. Just as Calder\'{o}n commutators appear very naturally in many applications in PDEs and the bilinear Hilbert transform also appears in applications, such as the AKNS systems \cite{MTT3}, it is natural to anticipate that operators of a similar type as Wu faces, but with an average dropped, will appear. Thus it is of interest to obtain $L^p$ estimates for operators of the following type

$$ (A, b, f) \mapsto p.v. \int\limits_{\mathbb{R}} F\left( \frac{A(x+t)-A(x)}{t} \right) b(x+\beta t)f(x+t)\frac{1}{t}\ dt $$

\noindent where $F$ is an analytic function. The first step would be to obtain $L^p$ estimates for

$$ (A, b, f) \mapsto p.v. \int\limits_{\mathbb{R}} \left( \frac{A(x+t)-A(x)}{t} \right)^m \ b(x+\beta t)f(x+t)\frac{1}{t}\ dt  $$

\noindent with polynomial bounds in $m$. Theorem \ref{original_thm} is the first step in showing a wide range of $L^p$ estimates for such operators when $m=1$.

\end{document}